\documentclass{amsart}
  \RequirePackage{amsmath}
\RequirePackage{amssymb} \RequirePackage{amsthm}
\RequirePackage{epsfig} 
                  \def\z{\zeta}
                \def\8{\theta}

\newcommand{\hol}{{\mathcal Hol}}
\DeclareMathOperator{\og}{O}

\def \M{ M}
\def\D{{\mathbb D}}
\def\T{{\mathbb T}}
\def\C{{\mathbb C}}  \def\N{{\mathbb N}}

 \def\cb{{\mathcal B}}

\def\Dp{{\mathcal D^p_{p-1}}}
\def \dph{H^\infty \cap \mathcal{D}_{p-1}^{p}}
\def \dqh{H^\infty \cap \mathcal{D}_{q-1}^{q}}
\def \dpb{\mathcal{B} \cap \mathcal{D}_{p-1}^{p}}
\def \dqb{\mathcal{B} \cap \mathcal{D}_{q-1}^{q}}
\def\Dpbmoa{{\mathcal D^p_{p-1}}\cap BMOA}
\def\Dqbmoa{{\mathcal D^q_{q-1}}\cap BMOA}

\def\Dpa{{\mathcal D^p_{\alpha}}}

\def\Dq{{\mathcal D^q_{q-1}}}

\def\({\left(}       \def\){\right)}

\newcommand{\iii}{\frac 1{2\pi}\int_0^{2\pi}}
\newcommand{\ig}{\stackrel{\text{def}}{=}}

\newcommand{\equstart}{\begin{equation}\begin{aligned}}
\newcommand{\equend}{\end{aligned}\end{equation}}
\newcommand{\equstartu}{\begin{equation*}\begin{aligned}}
\newcommand{\equendu}{\end{aligned}\end{equation*}}
\newtheorem{theorem}{Theorem}
\newtheorem{lemma}{Lemma}

\newtheorem{proposition}{Proposition}
\theoremstyle{definition}

\newtheorem{question}{Question}
\theoremstyle{remark}

\numberwithin{equation}{section}
\theoremstyle{theorem}
\newtheorem{other}{\bf Theorem}              % Other papers' theorems
\newtheorem{otherp}{\bf Proposition}  % Other papers' propositions
\newenvironment{pf}{\noindent{\emph{Proof.}}}{$\Box$ }
\newenvironment{Pf}{\noindent{\emph{Proof of}}}{$\Box$ }
\begin{document}
%%%%%%%%%%%%%%%%%%%%%%%%%%%%%%%%%%%% Title
\title[Multipliers of Dirichlet subspaces of the Bloch  space]
{Multipliers of Dirichlet subspaces of the Bloch  space}

%%% Information for the first author
\author[Ch.~Chatzifountas]{Christos Chatzifountas}
\address{Departamento de An\'alisis Matem\'atico,
Universidad de M\'alaga, Campus de Teatinos, 29071 M\'alaga, Spain}
 \email{christos.ch@uma.es}
%%% Second author
\author[D.~Girela]{Daniel Girela}
 \address{Departamento de An\'alisis Matem\'atico,
Universidad de M\'alaga, Campus de Teatinos, 29071 M\'alaga, Spain}
 \email{girela@uma.es}
%%%%%%%%%%%%%%%%%%%% Third author
\author[J.~A.~Pel\'aez]{Jos\'e \'Angel Pel\'aez}
 \address{Departamento de An\'alisis Matem\'atico,
Universidad de M\'alaga, Campus de Teatinos, 29071 M\'alaga, Spain}
 \email{japelaez@uma.es}

%%% General info
\date{November 20, 2012}
\begin{abstract}
For $0<p<\infty $ we let $\Dp$ denote the space of those functions
$f$ which are analytic in the unit disc $\D $
%in $\mathbb C$
 and
satisfy $\int_\D (1-\vert z\vert)\sp {p-1}\vert f'(z)\vert \sp
p\,dA(z)<\infty $.
\par It is known that, whenever $p\neq q$,  the only
multiplier from $\Dp $ to $\Dq $ is the trivial one. However, if $X$
is a subspace of the Bloch space and $0<p\le q<\infty$, then $X \cap
\Dp \subset X\cap \Dq $, a fact which implies that the space of
multipliers $\M(\Dp\cap X, \Dq\cap X)$ is non-trivial.
\par In this paper we study the spaces of multipliers $\M(\Dp\cap X, \Dq\cap
X)$ ($0<p,q<\infty $) for distinct classical subspaces $X$ of the
Bloch space. Specifically, we shall take $X$ to be $H^\infty $,
$BMOA$ and the Bloch space $\cb $.
\end{abstract}
\thanks{This research is supported by a grant from la Direcci\'{o}n General de Investigaci\'{o}n, Spain
(MTM2011-25502) and by a grant from la Junta de Andaluc\'{\i}a
(P09-FQM-4468 and FQM-210). The third author is supported also by
the \lq\lq Ram\'on y Cajal program\rq\rq , Spain.} \maketitle

%%%%%%%%%%%%%%%%%%%%%%%%%%%%%%%%%%%%%%%%%%%%%%%%%%%%%%%%%%%%%%%%%%%%%%%%%%%%%%%%%%%%%%%%%
\bigskip
\section{Introduction and main results}\label{intro}

Let $\D=\{z\in\C: |z|<1\}$ denote the open unit disc in the complex
plane $\C$ and let $\hol (\D)$ be the space of all analytic
functions in $\D$ endowed with the topology of uniform convergence
in compact subsets.
\par
If $\,0<r<1\,$ and $\,f\in \hol (\D)$, we set
$$
M_p(r,f)=\left(\iii |f(re^{it})|^p\,dt\right)^{1/p}, \,\,\,
0<p<\infty ,
$$
$$
M_\infty(r,f)=\sup_{\vert z\vert =r}|f(z)|.
$$
\par Whenever $\,0<p\le \infty $\, the Hardy space $H^p$ consists of those
$f\in \hol(\mathbb D)$ such that $\Vert f\Vert _{H^p}\ig
\sup_{0<r<1}M_p(r,f)<\infty $ (see \cite{D} for the theory of
$H^p$-spaces). If $0<p<\infty $ and $\alpha>-1$, the weighted
Bergman space $A^p_\alpha$ consists of those $f\in \hol(\mathbb D)$
such that
\[
\Vert f\Vert _{A^p_\alpha }\ig \left ((\alpha +1) \int_\mathbb
D(1-\vert z\vert )\sp\alpha \vert f(z)\vert ^p\, dA(z) \right
)^{1/p}<\infty .\] The unweighted Bergman space $A\sp p\sb 0 $ is
simply denoted by $A\sp p $. Here, $dA(z) =\frac{1}{\pi}dx\,dy $
denotes the normalized Lebesgue area measure in $\mathbb D$. We
refer to \cite{DS}, \cite{HKZ} and \cite{Zhu} for the theory of
these spaces.
\par
The space $\Dpa$ ($0<p<\infty $, $\alpha >-1$) consists of those
$f\in \hol(\mathbb D)$ such that $f'\in A^p_\alpha $. Hence, if $f$
is analytic in $\mathbb D$, then $f\in \Dpa$ if and only if
\[
\Vert f\Vert _{\Dpa}^p\ig \vert f(0)\vert^p+\Vert f'\Vert
_{A^p_\alpha }^p <\infty .\]
 If $p<\alpha +1$ then it is well known
that $\Dpa =A^p_{\alpha -p}$
 (see, e.\,\@g. Theorem\,\@6 of
\cite{Flett}). On the other hand, if $p>\alpha +2$ then $\Dpa
\subset H^\infty $. Therefore $\Dpa$ becomes a \lq\lq proper
Dirichlet space\rq\rq\, when $\alpha +1\le p\le \alpha +2$. The
spaces $\Dp $ are closely related with Hardy spaces. Indeed, it is
well known that $D^2_1=H^2$. We have also \cite{LP}
\begin{equation}\label{1}
H\sp p\subsetneq\Dp,\quad\text{ for $2\le p<\infty$},
\end{equation}
and \cite{Flett, Vi}
\begin{equation}\label{2}
\Dp\subsetneq H^p,\quad\text{ for $0<p\le 2$.}
\end{equation}
We remark that for $p\neq q$ there is no relation of inclusion
between $\Dp$ and $\Dq$ (see, e.\,\@g., \cite{BGP} and \cite{GP2}).
 \par We recall that the Bloch space $\cb$  consists of
those $f\in\hol(\D)$ such that
\begin{displaymath}
\|f\|_{\mathcal{B}}=|f(0)|+\sup_{z\in\D}(1-|z|^2)\,|f'(z)|<\infty.
\end{displaymath}
We refer to \cite{ACP} for the theory of Bloch functions.
\par\medskip Next, we consider multiplication operators. For $g\in\hol (\D )$, the multiplication operator $M_g$ is defined
by \[ M_g(f)(z)\ig g(z)f(z),\quad f\in \hol (\D),\,\, z\in \D.
\]
If $X$ and $Y$ are two normed (or Fr\'{e}chet) spaces of analytic
functions in $\D $ which are continuously contained in $\hol (\D )$,
$\M(X,Y)$ will denote the space of multipliers from $X$ to $Y$,
 $$\M (X,Y)=\{g\in\hol(\D):\, fg\in Y,\quad\text{for all $f\in
 X$}\},$$
 and $||M_g||_{(X\rightarrow Y)}$ will denote the norm of the operator $M_g$.
 If $X=Y$ we simply write $\M (X)$.
These operators have been studied on the Dirichlet type spaces
$\Dpa$ in \cite{GP:IE06,GP:JFA06, GaGiPe}, where among other results
it is proved
 that
 \begin{equation}\label{MDpDqtrivial}
 \M(\Dp,\Dq)=\{0\},\quad 0<p, q < \infty ,\,\, p\neq q.\end{equation}
\par\medskip The following simple observation plays an important
role in the motivation of this work.
\begin{lemma}\label{inclusionDpB} Suppose that \,$0<p<q<\infty $ and
$f\in \Dp \cap \mathcal B$. Then $f\in \Dq$.
\end{lemma}
\begin{proof} Since $f\in \mathcal B$ we have that
$\sup_{z\in \D}(1-\vert z\vert )\vert f^\prime (z)\vert =M<\infty $.
Using this we obtain
\begin{eqnarray*}\int_{\D }(1-\vert z\vert
)^{q-1}\vert f^\prime (z)\vert^q\,dA(z)\,& = \int_{\D
}\left[(1-\vert z\vert )\vert f^\prime (z)\vert \right
]^{q-p}(1-\vert z\vert )^{p-1}\vert f^\prime (z)\vert^p\,dA(z)\\&
\le M^{q-p}\int_{\D }(1-\vert z\vert )^{p-1}\vert f^\prime
(z)\vert^p\,dA(z)\,<\infty .\end{eqnarray*} Hence, $f\in \Dq$. $\qed
$
\end{proof}
Consequently, we have: \par {\it If\, $X$\, is a subspace of the
Bloch space then
\begin{equation}\label{X}
X\cap\Dp\subset X\cap\Dq ,\quad \text{if\,\, $0<p\le q<\infty$},
\end{equation}
a fact which, contrary to (\ref{MDpDqtrivial}), implies that
whenever $0<p\le q<\infty$, the space of multipliers $\M(\Dp\cap X,
\Dq\cap X)$ is non-trivial.}
\par\medskip
If $X\subset \cb$, the space $X\cap\Dp$ is equipped with the norm
$$ \|f\|_{X\cap \Dp} \ig\|f\|_{X} + \|f\|_ {\Dp}.$$
Our aim is this paper is to obtain a characterization of the spaces
$\M(\Dp\cap X, \Dq\cap X)$ $(0<p,q<\infty )$ for some important
subspaces $X$ of the Bloch space.
\par\medskip Let us start with $X=\cb $. For $\alpha>0$, the $\alpha$-logarithmic-Bloch space  $\cb_{\log,\alpha}$  consists of those
$g\in \hol (\D)$ such that
$$\rho_{\alpha}(f)\ig\sup_{z\in\D}(1-|z|^2)|g'(z)|\left(\log\frac{e}{1-|z|^2}\right)^\alpha<\infty .$$
It is clear that
\begin{equation}\label{eq:logemb}
\cb_{\log,\alpha}\subset \cb_{\log,\beta},\quad\alpha\ge\beta.
\end{equation}
For simplicity, the space  $\cb_{\log,1}$ will be denoted by $\cb_{\log}$.
\par The multipliers of the Bloch space into itself were characterized
independently by several authors (see \cite{Arazy82, BS, ZhuMult}).
Namely, we have the following result:
\begin{equation}\label{MB}\M(\cb )=\cb_{\log}\cap
H^\infty. \end{equation}
\par\medskip Let us turn our attention to the spaces
$\M(\Dp\cap\cb, \Dq\cap\cb )$. Among other results, we shall prove
that, for $p>1$, the space $\M(\dpb)$ coincides $\M(\cb)$. This is
part of the following result.
\begin{theorem}\label{th:blochdp}
Let \,$0<p,q<\infty$ and $g\in\hol(\D)$.
\par (i) If \,$1<q$ and $0<p \leq q <\infty$, then,  $$\M(\dpb,\dqb)=\M(\cb).$$
\par (ii) If \,$0<q<p<\infty$, then $$\M(\dpb, \dqb)=\{0\}.$$
\end{theorem}
\par The question of obtaining a complete characterization of $M(\dpb,\dqb)$ in the case $0<p\le q\le 1$ remains open.
However, we remark that the inclusion $$\M(\dpb,\dqb)\subset
\M(\cb),$$ is true for any $p, q$ (see the proof of
Theorem~\@\ref{th:blochdp} in Section~\ref{multbloch}). Using this,
the
 fact that $\M(\dpb,\dqb)\subset \dqb$,  and the
following result we see that part (i) of Theorem \ref{th:blochdp}
does not remain true for $0<q\le 1$.
\begin{theorem}\label{pr:1}
If \,$0<q\le 1$,  then   $\M(\cb)\setminus\Dq\neq\{0\}$.
\end{theorem}
\par\medskip
Let us now consider the spaces $\M (\dph,\dqh)$. It is easy to prove
the following result for the case $p\le q$.
\begin{theorem}\label{Hinfty:pleq1}
If $0<p\le q<\infty$, then $\M (\dph,\dqh)=\dqh$.
\end{theorem}
\par\medskip Regarding the case $0<q<p$, let us notice that
if $2\le q<p$ then $H^\infty \cap \Dp=H^\infty \cap \Dq=H^\infty $.
Hence we have
\begin{equation}\label{2leq<p}\M (\dph,\dqh)=H^\infty ,\quad 2\le
q<p.\end{equation}
\par\medskip
When \,$0<q<p$\, and\, $0<q<2$\, the question is more complicated.
It is well known (see \cite[Theorem\,\@1]{Gi92} and \cite{Vi}) that,
whenever $0<q<2$, there exists a function $f\in H^\infty \setminus
\Dq$. We improve this result in our next theorem.
\begin{theorem}\label{Hinfty:q<p1}
If\, $0<q<\min\{p,2\}$, then there exists a function $\,f\in \left(
\dph\right )\setminus \left (\dqh \right )$.
\end{theorem}
\par The functions constructed in Theorem~\ref{Hinfty:q<p1} are used in a basic way in the
proof of part~(a) of our following result.
\begin{theorem}\label{Hinfty:q<p2}$ $
\begin{itemize}
\item[(a)] If\, $0<q<1$ and $0<q<p<\infty $ then $\M (\dph,\dqh)=\{
0\}$.
\item[(b)] If\, $1\le q<2\le p$ then $\M (\dph,\dqh)=\{
0\}$.
\end{itemize}
\end{theorem}
\par\medskip
 In order to prove part $(b)$, we use strongly \cite[Theorem $1$]{Gi92} which asserts
 that, whenever\, $0<q<2$,
there exists a function $f\in H^\infty $ such that
\begin{equation}\label{aeradial}\int_0^1(1-r)^{q-1}\vert f^\prime (re^{i\theta })\vert
^q\,dr=\infty ,\quad \text{for almost every $\theta\in \mathbb
R$}.\end{equation} The case $1\le q<p<2$ of Theorem
\ref{Hinfty:q<p2} remains open.  However, if the answer to the
following open question were affirmative then it would follow that
the space $\M (\dph,\dqh)$ would be trivial also for this range of
parameters. (See the proof of Theorem\,\@\ref{Hinfty:q<p2}\,\@(b)).
\begin{question}\label{q1} Suppose that $0<q<p<2$. Does there exist a function $f\in \dph$ satisfying
(\ref{aeradial})?
\end{question}
\par\medskip
We end up taking $X=BMOA$, the space of those functions $f\in H^1$
whose boundary values have bounded mean oscillation on the unit
circle $\partial\mathbb D $ as defined by John and Nirenberg
\cite{JN}. A lot of information about the space $BMOA $ and can be
found in \cite{Ba, Gar, G:BMOA}. Let us recall here that
$$H^\infty \subsetneq BMOA\subsetneq \mathcal B,\quad\text{and}\quad
H^\infty \subsetneq BMOA\subsetneq \cap_{0<p<\infty }H^p.
$$
We emphasize also that $BMOA$ can be characterized in terms of
Carleson measures. If $I\subset
\partial\D$ is an interval, $\vert I\vert $ will denote the length
of $I$. The \emph{Carleson box} $S(I)$ is defined as
$S(I)=\{re^{it}:\,e^{it}\in I,\quad 1-\frac{|I|}{2\pi }\le r <1\}$.
If \,$\mu$ is a positive Borel  measure in $\D$, we shall say that
$\mu $
 is a Carleson measure if there exists a positive constant $C$ such that
\[
\mu\left(S(I)\right )\le C{|I|}, \quad\hbox{for any interval
$I\subset\partial\D $}.
\]
We have (see, e\,\@.g. \cite[Theorem\,\@6.\,\@5]{G:BMOA}):
\par\medskip {\it A function $f\in \hol (\D )$ belongs to $BMOA$ if and only
if the Borel measure \,$\mu _f$\, in\, $\D $\, defined by\,
$d\mu_f(z)=(1-\vert z\vert^2)\vert f^\prime (z)\vert ^2\,dA(z)$\, is
a Carleson measure.}
\par\medskip The multipliers of the space $BMOA$ have been characterized in
\cite{OF} (see also \cite{SiZh} and \cite{Zhao}). Indeed, we have
\begin{equation}\label{MBMOA}\M(BMOA)=H^\infty \cap BMOA_{\log}
.\end{equation} Here, $BMOA_{\log}$ is the space of those functions
$g\in H^1$ for which there exists a positive constant $C$ such that
\[\int _{S(I)}(1-\vert z\vert ^2)\vert g^\prime (z)\vert ^2 \,dA(z)\le
C\vert I\vert \left (\log\frac{2}{\vert I\vert }\right )
^{-2},\quad\hbox{for any interval $I\subset\partial\D $}.
\]
\par
Let us mention that $BMOA_{\log }$ is called $LMOA$ in \cite{SiZh}.
Following the terminology of \cite{Zhao}, we have:
\par\medskip {\it $BMOA_{\log}$ is the space of those functions
$g\in H^1$ for which the Borel measure \,$\mu _g$\, in\, $\D $\,
defined by\, $d\mu_g(z)=(1-\vert z\vert^2)\vert g^\prime (z)\vert
^2\,dA(z)$\, is a $2$-logarithmic Carleson measure.}
\par\medskip
In order to make a proper study of the spaces of multipliers  $\M
(\Dpbmoa ,\Dqbmoa )$, we shall present in
sections\,\@\ref{Facts-BMOAlog} and \ref{Random-ps} a series of
results concerning the space $BMOA_{\log }$, some of which are of
independent interest.
\par In section\,\@\ref{Facts-BMOAlog} we shall prove directly that
$BMOA_{\log }\subsetneq \cb_{\log }\subsetneq BMOA$ and we shall
also find some simple conditions on a function $f\in \hol (\D )$
which implies its membership to $BMOA_{\log }$. As a corollary we
shall prove the following result about lacunary power series in
$BMOA_{\log }$.
\begin{proposition}\label{lac-bmoalog-log3}
Let $f\in \hol (\D )$ be given by a lacunary power series,
i.\,\@,e., $f$ is of the form
$$\hbox{$f(z)=\sum_{k=0} \sp\infty a_kz\sp {n_k}$\, ($z\in \D $)\,
with $n_{k+1}\ge \lambda n_k$\, for all $k$, for a certain $\lambda
>1$.}$$
If\, $\sum_{k=0}^\infty \vert a_k\vert ^2(\log n_k)^3<\infty $, then
$f\in BMOA_{\log }\cap H^\infty $.
\end{proposition}
\par\medskip
Section\,\@\ref{Random-ps} deals with random power series of the
form
$$
f_t(z)=\sum_{n=0}^\infty r_n(t)a_nz^n,\quad z\in \D ,\quad 0\le t\le
1,$$ where $f(z)=\sum_{n=0}^\infty a_nz^n$ is analytic in $\D $ and
 $\{ r_n\} _{n=0}^\infty $ is the sequence of Rademacher function
 (see Section~\ref{preliminaries}).
Among other results, we establish a sharp condition on the Taylor
coefficients $a_n$ of $f$ which implies the almost sure membership
of $f_t$ in $BMOA_{\log }$.
\begin{theorem}\label{random.likesledd-bmoalog}
\begin{itemize}
\item[(i)] If\, $\sum_{n=1}^\infty \vert a_n\vert ^2(\log n)^3<\infty $
then for almost every $t\in [0, 1]$, the function
$$f_t(z)=\sum_{n=1}^\infty r_n(t)a_nz^n,\quad z\in \D,$$ belongs to
$BMOA_{\log }\cap H^\infty $. \item[(ii)] Furthermore, (i) is sharp
in a very strong sense: Given a decreasing sequence of positive
numbers $\{ \delta_n \} _{n=1}^\infty $ with $\delta _n\to 0$, as
$n\to \infty $, there exists a sequence of positive numbers $\{
a_n\} _{n=1}^\infty $ with $\sum_{n=1}^\infty a_n^2\delta_n(\log
n)^3<\infty $ such that, for almost every $t$ the function $f_t$
defined by $f_t(z)=\sum_{n=1}^\infty r_n(t)a_nz^n$ ($z\in \D$) does
not belong to $\cb_{\log }$. \end{itemize}
\end{theorem}
\par\medskip
Now we pass properly to study the multipliers from $\Dpbmoa$ to
$\Dqbmoa$ ($0<p, q<\infty $).
\par
 If\, $\lambda \ge 2$\, then $BMOA\subset \mathcal
D^\lambda_{\lambda -1}$. Hence, trivially, we have
\begin{equation}
\label{MBMOApqge2} \M (\Dpbmoa ,\Dqbmoa )=\M(BMOA)=BMOA_{\log}\cap
H^\infty ,\quad 2\le p, q<\infty .
\end{equation}
This remains true for
other values of $p$ and $q$.
\begin{theorem}\label{BMOA1<qpleq}
If $1<q<\infty $ and $0<p\le q<\infty $,\, then
$$\M (\Dpbmoa ,\Dqbmoa )=\M(BMOA)=BMOA_{\log}\cap H^\infty .$$
\end{theorem}
\par\medskip When $q<p$ then $0$ is the only multiplier from $\Dpbmoa $ to
$\Dqbmoa $, except in the cases covered by (\ref{MBMOApqge2}).
\begin{theorem}\label{BMOAq<p} If\, $0<q<p<\infty $ and $q<2$, then $$\M (\Dpbmoa ,\Dqbmoa
)=\{ 0\} . $$
\end{theorem}
\par\medskip
To deal with the remaining case, $0<p\le q\le 1$, we shall use the
above mentioned results about lacunary power series and random power
series. Our main results concerning random power series and
multipliers are contained in the following theorem.
\begin{theorem}\label{randon-bmoalogMdp}
Let $\{ a_n\} _{n=0}^\infty $ be a sequence of complex numbers
satisfying
\begin{equation}\label{rand-log3}\sum_{n=1}^\infty \vert a_n\vert ^2(\log
n)^3<\infty .\end{equation} For $t\in [0,1]$ we set
\begin{equation}\label{ft}f_t(z)=\sum_{n=0}^\infty
r_n(t)a_nz^n,\quad z\in \D ,\end{equation} where the ${r_n}'s$ are
the Rademacher functions. Then, for almost every $t\in [0,1]$, the
function $f_t$ satisfies the following conditions:
\begin{itemize}
\item[(i)] $\int_0^1(1-r)\left (\log \frac{1}{1-r}\right )^2\left [
M_\infty (r, {f_t}^\prime )\right ]^2\,dr<\infty $.
\item[(ii)] $f_t\in BMOA_{\log }\cap H^\infty $.
\item[(iii)] $f_t\in M(\Dpbmoa , \Dqbmoa )$ whenever $0<p\le q$ and
$q>\frac{1}{2}$.
\end{itemize}
\par Furthermore, if $0<q<\frac{1}{2}$ then there exists a
sequence $\{ a_n\} $ which satisfies (\ref{rand-log3}) and such that
$f_t\notin \Dq$, for almost every $t$. Thus, for this sequence $\{
a_n\} $ and for almost every $t$ we have:
\begin{itemize}
\item[(a)] $f_t\in \M(BMOA )$.
\item[(b)]  If\, $0<p\le \lambda $ and
$\lambda >\frac{1}{2}$\, then $f_t\in M(\Dpbmoa , \mathcal
D^\lambda_{\lambda -1}\cap BMOA)$.
\item[(c)]  $f_t\notin M(\Dpbmoa , \Dqbmoa )$ whenever $0<p\le q$.
\end{itemize}
\end{theorem}\par\medskip
We remark that Theorem\,\@\ref{randon-bmoalogMdp} shows that
Theorem\,\@\ref{BMOA1<qpleq} does not remain true for $q<1/2$.
\par\medskip
Finally, we  turn to consider multipliers in $\M (\Dpbmoa ,\Dqbmoa
)$ given by power series with Hadamard gaps.  We will show that
whenever $0<p\le q\le 1$ the power series with Hadamard gaps in $\M
(\Dpbmoa ,\Dqbmoa )$ coincide with those in $\Dq \cap BMOA_{\log }$
and will obtain also the analogue of
Theorem\,\@\ref{randon-bmoalogMdp} for lacunary power series in
Theorem\,\@\ref{lac-bmoalogMdp}. This will give another proof of the
impossibility of extending Theorem\,\@\ref{BMOA1<qpleq} to $q<1/2$.

\par\bigskip

\section{Preliminary results}\label{preliminaries}
As usual, a sequence of positive integers $\{ n_k\}_{k=0}^\infty $
is said to be lacunar if there exists $\lambda >1$ such that
$n_{k+1}\ge \lambda n_k$, for all $k$. Also,
 by a {\it lacunary power series\/} (also called {\it power series with Hadamard
gaps\/}) we mean a power series of the form
$$\hbox{$f(z)=\sum_{k=0} \sp\infty a_kz\sp {n_k}$\, ($z\in \D $)\,
with $n_{k+1}\ge \lambda n_k$\, for all $k$, for a certain $\lambda
>1$.}$$ For simplicity, we shall let $\mathcal L$ denote the
class of all function $f\in\hol (\D )$ which are given by a lacunary
power series. Several known results on power series with Hadamard
gaps will be repeatedly used along the paper, we collect them in the
following statement, (see \cite{BKV,Zyg, ACP}).
\begin{otherp}\label{Had} Suppose that \,$0<p<\infty $, $\alpha >-1$
and  $f$ is an analytic function in $\D $ which is given by a
power series with Hadamard gaps,
$$\hbox{$f(z)=\sum_{k=0}
\sp\infty a_kz\sp {n_k}$\, ($z\in \D $)\, with $n_{k+1}\ge \lambda
n_k$\, for all $k$ \,($\lambda >1$).}$$ Then:
\par (i)\,
$f\in \Dpa \quad \Longleftrightarrow \quad \sum_{k=0}\sp \infty
n_k^{p-\alpha -1} \vert a_k\vert \sp p <\infty $, and
$$||f-f(0)||^p_{\Dpa }\asymp \sum_{k=0}\sp \infty n_k^{p-\alpha -1}
\vert a_k\vert \sp p.$$
\par (ii)\, $f\in H^\infty$ if and only if $ \sum_{k=0}\sp \infty
\vert a_k\vert  <\infty $, and
$$||f||_{H^\infty}\asymp \sum_{k=0}\sp \infty
\vert a_k\vert.$$
\par (iii) $f\in\cb \Longleftrightarrow \quad \sup_{n}|a_n|<\infty$, and $$\|f\|_{\cb}\asymp \sup_{n}|a_n|.$$
\end{otherp}
\par
It is also well known that $\mathcal L\cap H^p=\mathcal L\cap H^2$
for any $p\in (0, \infty )$ but $$H^\infty \cap \mathcal L\subsetneq
H^2\cap \mathcal L.$$ In spite of this, for any given lacunary
sequence of positive integers $\{ n_k\}_{k=1}^\infty $ and any
sequence of complex numbers $\{ u_k\}_{k=1}^\infty \in \ell ^2$,
Fournier constructed in \cite{fournier} a function
$f(z)=\sum_{n=0}^\infty a_nz^n\in H^\infty $ with $a_{n_k}=u_k$, for
all $k$. Some properties of the bounded function $f$ which were not
stated in \cite{fournier} will play an important role in the proof
of some of our results. Due to this fact and for sake the
completeness we present a complete proof of Fournier's construction
pointing out some extra properties of the constructed function (for
simplicity we shall restrict to sequences $\{ n_k\} $ satisfying
$n_{k+1}\ge 2n_k$).
\par Let start fixing some notation. The unit circle $\partial \D$
will be denoted by $\mathbb T$. If $g\in L^1(\mathbb T)$ its Fourier
coefficients $\hat g(n)$ are defined by
\begin{equation*}\hat g(n)=\frac{1}{2\pi }\int_{-\pi }^\pi
g(e^{i\theta })e^{-in\theta }\,d\theta ,\quad n\in \mathbb
Z.\end{equation*}
\par
If $n_1<n_2$ are integers we shall write $\lfloor n_1, n_2\rfloor$
for the set of all integers $n$ with $n_1\le n\le n_2$. Also, for
$g(z)=\sum_{k=0}^\infty b_k z^k\in\hol(\D)$ and $n_2>n_1\ge 0$, we
set
    $$
    S_{n_1,n_2}g(z)=\sum_{k=n_1}^{n_2-1}b_kz^k.
    $$

\begin{proposition}\label{le:fournier}
Assume that $\{u_k\}_{k=0}^\infty\in \ell^2$  and let
$\{n_k\}_{k=0}^\infty$ be a sequence of positive integers such that
$n_{k+1}
> 2n_k$, for all $k$. Then, there exists a function $\Psi\in\hol (\D )$ of the
form
$$\Psi(z)=\sum _{n=0}^\infty a_nz^n, \quad z\in \D ,$$ with
the following properties: \vspace{5mm}
\begin{itemize}
\item  [(i)]  $\Psi \in H^\infty$.
\item [(ii)] $a_{n_k}=u_k$, for all $k$.
\item [(iii)] If we define $\Lambda _0=\{n_0\}$ and $\Lambda _k=\lfloor n_{k}-n_{k-1}, n_{k}\rfloor$ for $k>0$, we have that
the sets $\Lambda _k$ are pairwise disjoint and satisfy $\Lambda
_k\subset \lfloor n_{k-1}+1, n_k\rfloor$ for all $k\ge 1$.
Furthermore, $a_n= 0$ if $n \not \in \cup_{k=0}^\infty\Lambda _ k$.
\item [(iv)] There is an absolute constant  $C$ such that $$\left\Vert S_{n_k+1,n_{k+1}+1}\Psi\right\Vert_{H^\infty}\le C |u_k|,\quad
\text{for all $k$}.$$
\end{itemize}
\end{proposition}

\begin{proof}
The construction depends on the following equality
\cite[p.\,\@402]{fournier}
\begin{equation}\label{equality}\vert a+vb\vert ^2+\vert b-\bar
va\vert =(1+\vert v\vert^2)(\vert a\vert^2+\vert b\vert ^2),\quad a,
b, v\in \mathbb C.\end{equation}
 Let us define inductively the following
sequences of functions on $\T$
\begin{equation}\label{phi0}
\phi_0(\zeta)=u_0\z^{n_0},\quad h_0(\zeta) = 1,\quad \z\in\T,
\end{equation}
and, for $k>0$,
\begin{equation}\label{phik}
\phi_k(\zeta)=\phi_{k-1}(\z)+u_k\z^{n_k}h_{k-1}(\z),\quad h_k(\zeta
)=h_{k-1}(\z)-\overline{u_k}\z^{-n_k}\phi_{k-1}(\z),\quad
(\z\in\T).\end{equation}
\par
Since $n_{k+1}>2n_k$, it is clear that the sets $\Lambda _k$,
$k=1,2,\dots$, are disjoint and that $\Lambda _k\subset \lfloor
n_{k-1}+1, n_k\rfloor$ for all $k\ge 1$.
\par
We claim that that the sequences $\{ \phi _k\}$ and $\{ h_k\} $
satisfy the following properties
\begin{equation}\label{p2}
\widehat{\phi_k}(n)=0,\quad \text{whenever $k\ge 0$ and $n \not \in
\bigcup_{j=0}^k\Lambda_{ j}$}
\end{equation}

\begin{equation}\label{p2p}
 \widehat{h_k}(-n)=0,\quad \text{whenever $k\ge 0$ and $n\ge 1$ and $n \not\in \bigcup_{j=1}^k\Lambda_{ j}$.}
\end{equation}
\begin{equation}\label{p3}
\widehat{\phi_k}(n)=\widehat{\phi_j}(n),\quad
 \text{whenever $k\ge j$ and $n\le n_j$, }
\end{equation}

\begin{equation}\label{p3p}
\widehat{\phi_k}(n_j)=u_j,\quad\text{whenever $k\ge j$ }.
\end{equation}
\par It is clear that (\ref{p2}) and (\ref{p2p}) hold for $k=0,1$. Arguing by induction, assume that
(\ref{p2}) and (\ref{p2p}) are valid for some value of $k\in\N$.
Then,
\begin{equation}\label{skatakaiaposkata} \phi_{k+1}(\z) = \phi_k(\z)+ u_{k+1}\z^{n_{k+1}}h_k(\z) =
 \sum_{j=0}^k\sum _ { n  \in\Lambda_j }\widehat{\phi_k}(n)\z^n+f_k(\z),
 \end{equation}
 where $f_k(\z)=u_{k+1}\z^{n_{k+1}}h_k(\z)$.
By the induction hypotheses $\widehat{f_k}(n)=0$ if
$n\notin\Lambda_{k+1}$, which gives (\ref{p2}) for $k+1$. The proof
of (\ref{p2p}) is analogous.
 Now, (\ref{p3}) follows from (\ref{phik}), (\ref{p2}) and  the fact that the sets $\Lambda_k$ are disjoint and (\ref{p2p}).
 Using  again that the sets $\Lambda_k$ are disjoint, (\ref{p3}), (\ref{phik}) and (\ref{phi0}), we deduce (\ref{p3p}).

\par\medskip We have that
$$|\phi_0(\zeta)|^2+|h_0(\zeta)|^2=1+|u_0|^2,$$ so  if we assume that
$|\phi_k(\zeta)|^2+|h_k(\zeta)|^2=\prod_{j=0}^k(1+|u_j|^2)$, bearing
in mind (\ref{equality}) and (\ref{phik}), it follows that
$$|\phi_{k+1}(\zeta)|^2+|h_{k+1}(\zeta)|^2=\left(1+|u_{k+1}|^2\right)\left(|\phi_k(\zeta)|^2+|h_k(\zeta)|^2\right)=\prod_{j=0}^{k+1}(1+|u_j|^2),$$
%which implies (\ref{p1}).
hence we have proved by induction that
\begin{equation*}\label{p1}|\phi_k(\zeta)|^2+|h_k(\zeta)|^2=\prod_{j=0}^k(1+|u_j|^2),\quad \zeta\in \mathbb T,\quad k=0, 1, 2, \dots
.\end{equation*} This and the fact that $\{u_k\}_{k=0}^\infty\in
\ell^2$ imply that $\{ h_k\} _{k=0}^\infty $ and $\{ \phi_k\}
_{k=0}^\infty $ are uniformly bounded sequences of functions in
$L^\infty (\mathbb T)$. Then, using the Banach-Alaoglu theorem,
(\ref{p2}), (\ref{p3}) and (\ref{p3p}), we deduce that a subsequence
of $\{ \phi _k\} $ converges in the weak star topology of $L^\infty
(\T )$ to a function $\phi\in L^\infty (\T )$ with $\hat\phi (n)=0$
for all $n<0$, and $\hat\phi (n_k)=u_k$ for all $k$. Then if we set
$a_n=\hat\phi (n)$ ($n\ge 0$) it follows that the function $\Psi $
defined by
$$\Psi (z)=\sum_{n=0}^\infty a_nz^n,\quad z\in \D,$$ is analytic in
$\D $ and satisfies (i), (ii) and (iii).
 \par Finally, we shall prove (iv). Using(\ref{p3}) and
 (\ref{skatakaiaposkata}), we see that for any $\zeta \in \T$, we
 have
 \begin{equation*}\begin{split}
 & S_{n_k+1,n_{k+1}+1}\Psi(\zeta)=\sum_{n=n_{k}+1}^{n_{k+1}}\widehat{\Psi}(n)\zeta^n=
 \sum_{m=n_k+1}^{n_{k+1}}\left(\lim_{m\to\infty}\widehat{\phi_{m}}(n)\right)\zeta^n
 \\ & =\sum_{n=n_{k}+1}^{n_{k+1}}\widehat{\phi_{k+1}}(n)\zeta^n =f_k(\zeta)=u_{k+1}\z^{n_{k+1}}h_k(\z),
 \end{split}\end{equation*}
 which, bearing in mind that $\sup_{k}\Vert h_k\Vert_\infty =C<\infty $, implies
 \begin{equation*}\begin{split}
 \left\Vert S_{n_k+1,n_{k+1}+1}\Psi\right\Vert_{H^\infty}
 =  |u_{k+1}|||h_k||_{L^\infty{(\T)}}
   \le C |u_{k+1}|.
 \end{split}\end{equation*}
 This finishes the proof.
 \end{proof}
 \par\medskip Our work will also make use
of  the Rademacher functions  $\{r_n(t)\}_{n=0}^\infty$ which are
are defined by
\[ r_0(t)=\left \{  \begin{array}{ll}
1,& \mbox{\,\,\,\,\,if\, $0<t<1/2$}\\
-1,&\mbox{\,\,\,\,\,if\, $1/2<t<1$}\\0,&\mbox{\,\,\,\,\,if\,
$t=0,1/2,1.$}\end{array}\right .
\]
\[
r_n(t)=r_0(2^nt),\quad n=1,2, \dots .
\]
See, e.~g., \cite[Chapter~V, Vol.~I]{Zyg} or \cite[Appendix~A]{D}
for the properties of these functions. In particular, we shall use
Khinchine's inequality which we state as follows.
\begin{otherp}[Khinchine's inequality]\label{Khinchine}
If $\{ c_k\} _{k=1}^\infty \in \ell ^2$ then the series $\sum
_{k=1}^\infty c_kr_k(t)$ converges almost everywhere. Furthermore,
for $0<p<\infty $ there exist positive constants $A_p, B_p$ such
that for every sequence $\{ c_k\} _{k=0}^\infty \in \ell ^2$ we have
$$A_p\left (\sum_{k=0}^\infty \vert c_k\vert ^2\right )^{p/2}
\le \int_0^1\left \vert\sum _{k=0}^\infty c_kr_k(t)\right \vert
^pdt\le B_p\left (\sum_{k=0}^\infty \vert c_k\vert ^2\right
)^{p/2}.$$
\end{otherp}

\bigskip
\section{Multipliers on $\dpb$}\label{multbloch}
\begin{Pf}{\,\em{Theorem \ref{th:blochdp}.}}\,
(i)\, Assume that $g\in \M(\dpb,\dqb)$. From now and throughout the paper we shall denote by  $\varphi _a$  the
M\"obius transformation which interchanges the origin and $a$,
$$\varphi_a (z)=\frac{a-z}{1-\overline{a}z},\quad z\in \D.$$
\par A simple calculation shows that
$$\sup_{a\in\D}||\varphi_a||_{\dpb}<\infty.$$
So, for any $a,z\in\D$
 \begin{equation}
 \begin{split}\label{nolabel}
 &(1-|z|^2)|\varphi_a '(z)g(z)|= (1-|z|^2)|(\varphi_a\cdot g)'(z) - \varphi_a (z)g'(z)| \\
 &\le \|\varphi_a  g\|_{\cb \cap \Dq} + (1-|z|^2)|\varphi_a (z)g'(z)|\lesssim\|\M _g\|_{(\cb \cap \Dp \rightarrow \cb \cap
 \Dq)}+||g||_{\cb}<\infty .
\end{split}
 \end{equation} Since $(1-\vert a\vert^2)|\varphi_a '(a)\vert =1$,
 taking $z=a$ in (\ref{nolabel}) we obtain $$\vert g(a)\vert \lesssim\|\M _g\|_{(\cb \cap \Dp \rightarrow \cb \cap
 \Dq)}+||g||_{\cb}<\infty ,$$ for any $a\in\D $. Thus, $g\in H^\infty$.
\par Next consider the family of test functions, $ f_{\theta}(z)=\log\frac{1}{1-ze^{-i\theta}}$, $\theta\in [0,2\pi)$. A calculation shows that
 $\{f_{\theta}\}_{\theta\in [0,2\pi)}$ is  uniformly bounded in  $\dpb$.
Therefore,
\begin{equation*}
\begin{aligned}
A &=\sup_{\theta\in[0,2\pi)}\|gf_\theta\|_ \cb \leq
\sup_{\theta\in[0,2\pi)}\|gf_\theta\|_{\cb \cap \Dq}
\\ & \leq
 \|\M _g\|_{(\cb \cap \Dp \rightarrow \cb \cap \Dq)}\sup_{\theta\in[0,2\pi)} \|f_\theta\|_{\cb \cap \Dp}<\infty,
 \end{aligned}
\end{equation*}
which implies that
\begin{equation*}
\begin{split}
(1-|z|^2)|g'(z)f_\theta(z)|&=(1-|z|^2)|g'(z)f_\theta(z)+
g(z)f_\theta '(z)-g(z)f_\theta '(z)|
\\ &\leq A +(1-|z|^2)|g(z)f_\theta '(z)|
\\&=A+||g||_{H^\infty}\sup_{\theta\in [0,2\pi)}|| f_{\theta}||_{\cb}
\\ &<\infty ,\quad\text{for all \, $z \in \D$  and $\theta\in[0,2\pi)$.}
\end{split}
\end{equation*}
Finally, given $z\in \D$  choose $e^{i\theta}=\frac{z}{|z|}$ to
deduce that
$$\sup_{z\in\D}|g'(z)|(1-|z|)\log\frac{1}{1-|z|}<\infty,$$
which together the fact that $g\in H^\infty$ gives that $g\in
\M(\cb)$.
\par\medskip Suppose now that $g\in \M(\cb)$ and take $f\in \cb \cap
\Dp$. Then $fg\in \cb $. Using Lemma~\ref{inclusionDpB} and the
closed graph theorem, we obtain
\begin{equation}
\begin{aligned}\label{eq:dpparthinfty}
&\int _\D |(fg)'(z)|^{q}(1-|z|^2)^{q-1} dA(z)
\\ & \lesssim \int _\D |f'(z)g(z)|^{q}(1-|z|^2)^{q-1}dA(z)
+ \int _\D |g'(z)f(z)|^{q}(1-|z|^2)^{q-1}dA(z)
\\ & \lesssim \|g\|_{H^\infty}^q\|f\|_{\dpb}^q
+ \int _\D |f(z)g'(z)|^{q}(1-|z|^2)^{q-1}dA(z).
\end{aligned}
\end{equation}
We shall distinguish two cases to deal with the last integral which
appears in (\ref{eq:dpparthinfty}). First, if $1<q\leq 2$, bearing in mind
(\ref{2}) and the fact that $g\in \cb_{\log}$, we see that
\begin{equation}
\begin{aligned}\label{eq:dppart2hinfty}
 \int _\D |f(z)g'(z)|^{q}(1-|z|^2)^{q-1}dA(z) \lesssim & \int _0 ^1  \frac{1}{(1-r)^q\log ^q \frac{e}{1-r}} M_q^q(r,f)(1-r)^{q-1}\,dr
  \\ \lesssim & \|f\|_{\Dq}^q  \int _0 ^1  \frac{1}{(1-r)\log ^q \frac{e}{1-r}} \,dr
\\  \lesssim & \|f\|_{\dpb}^q .
\end{aligned}
\end{equation} \par On the other hand,  if  $2< q< \infty$,  then using that $g\in \cb_{\log}$ and the well known fact that
  $$\M_q(r,f)\le C||f||_{\cb}\left(\log\frac{1}{1-r} \right)^{1/2},\quad 0<r<1,$$ (see, e.\,\@g.,  \cite{ClMg}) we get
\begin{equation}
\begin{aligned}\label{eq:dppart3hinfty}
 \int _\D |f(z)g'(z)|^{q}(1-|z|^2)^{q-1}dA(z) \lesssim &   \int _0 ^1  \frac{1}{(1-r)^q\log^q\frac{e}{1-r}} M_q^q(r,f)(1-r)^{q-1}\,dr
\\  \lesssim & ||f||^q_{\cb}   \int _0 ^1    \frac{1}{(1-r)\log^{q/2}\frac{e}{1-r}}\, dr < \infty.
\end{aligned}
\end{equation}
Joining (\ref{eq:dpparthinfty}) (\ref{eq:dppart2hinfty}) and (\ref{eq:dppart3hinfty}),
we see that in any case we have $fg\in \Dq$ and, hence, $fg\in
\cb\cap\Dq$. Thus, we have proved that $g\in \M(\cb\cap \Dp, \cb\cap
\Dq)$ finishing the proof.
 \medskip
\par (ii)\,
 We borrow ideas from \cite[Theorem 12]{GaGiPe}. We shall distinguish three cases.\vspace{5 mm} \newline
{\bf{\large{ Case 1. $\mathbf{2<q<\infty}.$}}} Assume that $g\in
\M(\dpb, \dqb)$ and $g \not\equiv 0$. By the proof of \cite[Theorem
K]{GaGiPe} (see also the proofs of \cite[Theorems $1.6$ and
$1.7$]{GP2}), it follows that there exists a function  $f\in\Dp$,
given by a lacunary power series, with $f(0)\neq 0$, and  such that
its sequence of ordered zeros $\{z_n\}$ (that is, the $z_n's$ are
ordered so that $\vert z_1\vert\le \vert z_2\vert \le \vert z_3\vert
\dots $) satisfies
$$\prod_{n=1} ^ N  \frac{1}{|z_n|} \neq o\Biggl(\log{N}\Biggr)^{\frac12 -\frac1q}.$$
Since $f$ is given by a lacunary power series, by Proposition
\ref{Had}, the sequence of its Taylor coefficients is  in $\ell^p$.
This implies that $f\in\dpb$.
 If $\{w_n\}$ is the sequence of non-zero zeros of $gf$ arranged so
 that $\vert w_1\vert\le \vert w_2\vert \le \vert w_3\vert
\dots $, we have that $|w_n| \leq |z_n|$, for all $n$, which gives
that
\begin{equation*}
\prod_{n=1} ^ N \frac{1}{|w_n|} \geq \prod_{n=1} ^ N
\frac{1}{|z_n|},
\end{equation*}
 hence
 \begin{equation*}
 \prod_{n=1} ^ N \frac{1}{|w_n|} \neq o\Biggl(\log {N}\Biggr)^{\frac12 -\frac1q}.
 \end{equation*}
This together with \cite[Theorem 1.6]{GP2} implies that $fg \notin
\Dq$. This is a contradiction. Thus, $g\equiv 0$.

\vspace{5mm}  {\bf{\large{Case 2. $\mathbf{0<q\leq 2 <p.}$}}} The
proof is similar to that of the case 1. Suppose that $g\not\equiv 0$
and $g\in \M(\dpb ,\dqb )$. Take
 $\gamma\in\left(0,\frac{1}{2}-\frac{1}{p}\right)$.
Then, by  the proof of \cite[Theorem K]{GaGiPe} and Proposition
\ref{Had}, there is a function  $f\in\dpb$, represented by a
lacunary series, with $f(0)\neq 0$ whose sequence of ordered  zeros
$\{z_n\}$ satisfies \begin{equation}\label{zerosgamma}\prod_{n=1} ^
N \frac{1}{|z_n|} \neq o\Biggl(\log{N}\Biggr)^{\gamma
}.\end{equation}
\par
 Let $\{w_n\}_{n=1}^\infty$ be the sequence of ordered
non-zero zeros of $fg$. Since $fg\in \Dq $ and $q\le 2$, it follows
that $fg\in H^q$ and, hence, $\{w_n\}_{n=1}^\infty$ satisfies the
Blaschke condition which is equivalent to saying that
\begin{eqnarray*}
\prod_{k=1}\sp N \frac{1}{\vert w_k\vert} =\og(1),\quad\hbox{as
$N\to\infty$.}
\end{eqnarray*}
This is in contradiction with (\ref{zerosgamma}), because any zero
of $f$ is also a zero of $fg$. Consequently,  $g\equiv 0$.

\vspace{5mm} {\bf{\large{ Case 3. $\mathbf{0<p\leq 2}.$}}} Suppose
that $g\not\equiv 0$ and $g \in \M (\cb \cap \Dp,\cb \cap \Dq)$.
Take $a_n=\frac{1}{n^{1/p +\varepsilon}}$ with $0<\varepsilon
<\tfrac1q -\tfrac1p$ and $f(z) = \sum _{n=1}^{\infty} a_n z^{2^n}$ .
Since $\sum _{n=1}^{\infty}a_n ^p < \infty $ and  $\sum
_{n=1}^{\infty}a_n ^q = \infty $, then by Proposition \ref{Had},
$f\in \dpb\setminus \Dq$.
\par Let $\{ r_k(t)\} $ be the Rademacher functions and let $f_t(z)=\sum _{k=1}^{\infty}r_k(t)a_kz^{2^k}$. By Proposition \ref{Had} (iii)
$$ \|f\|_{\cb}\asymp \sup_{n}|a_n|\asymp \|f_t\|_{\cb},\quad t\in[0,1]$$
and
\begin{equation}
\|f_t\|_{H^2}^{2p} = \left (\sum _{k=0}^{\infty} |a_k|^2\right)^p
\leq \left (\sum _{k=0}^{\infty} |a_k|^p\right)^2 \asymp
\|f_t\|_{\Dp}^{2p} \asymp \|f\|_{\Dp}^{2p},\quad t\in[0,1]
\end{equation}
Then for any $t\in[0,1]$,  it follows that
\begin{equation}
\int _ {\D} |(gf_t)'(z)|^q(1-|z|^2)^{q-1}dA(z) \lesssim
\|f_t\|_{\Dp} ^q + ||f_t||^q_{\cb}\asymp \|f\|_{\Dp} ^q
+||f||^q_{\cb}<\infty.
\end{equation}
So, by Fubini's theorem, Khinchine's inequality and the fact that
$g\in \Dq$, we obtain
\begin{equation}\label{eq:4}
\begin{aligned}
& \int_0^1 \int _ \D|gf'_t(z)|^q(1-|z|^2)^{q-1}dA(z)dt \\& \lesssim
\int_0^1 \int _ \D|(gf_t)'(z)|^q(1-|z|^2)^{q-1}dA(z)dt + \int_0^1
\int _ \D|f_tg'(z)|^q(1-|z|^2)^{q-1}dA(z)dt \\& \lesssim
\|f\|_{\dpb}^q +\int_ \D |g'(z)|^q \int _
0^1|f_t(z)|^q(1-|z|^2)^{q-1}dt dA(z) \\ & \lesssim \|f\|_{\dpb}^q
+\int_ \D |g'(z)|^q M^q_2(|z|,f)(1-|z|^2)^{q-1} dA(z) \\ & \lesssim
\|f\|_{\dpb}^q +\|f\|_{\Dp}^q \int_ \D |g'(z)|^q(1-|z|^2)^{q-1}
dA(z)
\\ & \lesssim \|f\|_{\dpb}^q.
\end{aligned}
\end{equation}
On the other hand, since $g\not\equiv 0$, there exists a  positive
constant $C$ such that $ M_q ^q (r,g) \geq C$, $1/2 <r<1$. Using
Fubini's theorem, Khinchine's inequality and bearing in mind that
$f'$ is also given by a power series with Hadamard gaps (thus
$M_2(r,f')\asymp M_q(r,f'))$ we have that

\begin{equation}
\begin{aligned}
& \int _ 0 ^1 \int _ \D |gf_t'(z)|^q(1-|z|^2)^{q-1}dA(z)dt \\ & =
\int _ \D |g(z)|^q(1-|z|^2)^{q-1}\left(\int _ 0 ^1 |f_t'(z)|^q
dt\right) dA(z) \\& \asymp \int _ \D
|g(z)|^q(1-|z|^2)^{q-1}M^q_2(|z|,f') dA(z) \\ & \geq C \int _{1/2}^1
M^q_q(r,g)M^q_q(r,f')(1-r^2)^{q-1}dr \\ & \geq C \int _{1/2}^1
M^q_q(r,f')(1-r^2)^{q-1}dr = + \infty .
\end{aligned}
\end{equation}
 This is in contradiction with
(\ref{eq:4}). It follows that $g\equiv 0$.
\end{Pf}

\par\medskip We remark that the argument used to prove the inclusion
$\M(\cb\cap\Dp , \cb\cap\Dq )\subset \M(\cb )$ in the proof of
Theorem~1~(i) works for any values of $p$ and $q$, that is we have
$$\M(\cb\cap\Dp , \cb\cap\Dq )\subset \M(\cb ),\quad 0<p, q<\infty .$$
We do not have a complete characterization of the space
$\M(\cb\cap\Dp , \cb\cap\Dq )$ in the case $0<p\le q\le 1$, however we
find a sharp sufficient condition on a function $g$ to lie in this
space of multipliers. We note that Theorem \ref{pr:1} is a byproduct of part (ii) of the following stronger result.
\begin{proposition}\label{pr:bloga}
Let $0<p\le q\le 1$, $\alpha\in\left(\frac{1}{q},\infty\right)$ and $g\in\hol(\D)$. Then,
\par $(i)$\, If $g\in \cb_{\log,\alpha}\cap H^\infty$, then $g\in\M(\dpb, \dqb)$.
\par $(ii)$\, $\left(\cb_{\log,\frac{1}{q}}\cap H^\infty\right)\setminus\Dq\neq\{0\}$.
\end{proposition}

\begin{proof}
Part (i) can be proved arguing as in \eqref{eq:dpparthinfty} and  \eqref{eq:dppart2hinfty}, so we omit a detailed proof.
\par (ii)\,  Assume first that $0<q<1$. Consider the lacunary power series $$g(z)=\sum _{k=1} ^ {\infty} \frac{1}{k^{1/q}} z^{2^k}.$$
\par
By Proposition \ref{Had}, $g\in H^\infty\setminus \Dq$. Since
$\limsup \limits _{k\rightarrow \infty}\frac{1}{k^{1/q}} \left(\log
2^k\right)^{1/q}< \infty$ (see \cite[p.\,\@20]{PelRathg})
$g\in\cb_{\log,\frac{1}{q}}$.
%Thus $g\in \M(\cb )\setminus \Dq$.
\par\medskip Let us consider now the case $
q=1$. The proof in this case is a little bit more involved.
Set
$$u_k=\frac{1}{k+1}\,\,\,\,\text{and}\,\,\,\,n_k=4^k,\quad  k=0, 1, 2,\dots .$$ Let $\Psi$ be
the $H^\infty $-function associated to these sequences via
Theorem~\ref{le:fournier}. By \cite[Lemma $1.6$ (i)]{Vi},
$$||\Psi||_{\mathcal{D}^1_0}\gtrsim \Vert \{\widehat{\Psi}(4^k\}_{k=0}^\infty \Vert_{\ell^1}=\sum_{k=0}^\infty\frac{1}{k+1}=\infty.$$
Finally, we shall see that $\Psi\in \mathcal{B}_{\log}$. Bearing in
mind, \cite[p.\,\@113]{Pabook}, \cite[Lemma\,\@3.\,\@1]{MatelPavstu}
and Lemma \ref{le:fournier} (iv), we deduce
\begin{equation*}\begin{split}
&M_\infty(r,\Psi')\le |\widehat{\Psi}(1)|+\sum_{k=0}^\infty
M_\infty(r,S_{n_{k}+1,n_{k+1}}\Psi')
\\ &\lesssim ||\Psi||_{H^\infty}+ \sum_{k=0}^\infty ||S_{n_{k}+1,n_{k+1}}\Psi'||_{H^\infty} r^{4^{k}}
\\ & \lesssim ||\Psi||_{H^\infty}+ \sum_{k=0}^\infty 4^{k}||S_{n_{k}+1,n_{k+1}}\Psi||_{H^\infty} r^{4^{k}}
\\ & \lesssim ||\Psi||_{H^\infty}+ \sum_{k=0}^\infty \frac{4^{k}}{k+1}r^{4^{k}}.
\end{split}\end{equation*}
Since an standard calculation shows that
$$ \sum_{k=0}^\infty \frac{4^{k}}{k+1}r^{4^{k}}\lesssim \frac{1}{(1-r)\log\frac{e}{1-r}},\quad 0\le r<1,$$
this finishes the proof.
   \end{proof}

\par\medskip
 Next we
provide a sufficient condition, which involves Carleson measures, on a function $g$ to lie in this
space of multipliers. It  turns out to be also necessary
if $g$ is given by a power series with Hadamard gaps.

\begin{theorem}\label{pq1}
Assume that $0<p\le q\le 1$ and let $g$ be an analytic function in
$\D $. Let $\mu_{g,q}$ be the Borel measure in $\D$ defined by
$d\mu_{g,q}(z) = |g'(z)|^q(1-|z|^2)^{q-1}\,dA(z)$.
\begin{itemize}\item[(a)] If $g\in H^\infty \cap \cb_{\log}$
and the measure $\mu_{g,q}$ is a Carleson measure, then $g \in
\M(\dpb,\dqb)$.
\item[(b)] If $g$ is given by a power series with
Hadamard gaps, then $g \in \M(\dpb,\dqb)$ if and only if $g\in
H^\infty\cap\cb_{\log}$ and the measure $\mu_{g,q}$ is a Carleson
measure.
\end{itemize}
\end{theorem}
\par\medskip
\begin{proof}
Suppose that $g\in H^\infty \cap \cb_{\log}$ and the measure
$\mu_{g,q}$ is a Carleson measure. Take $f\in \cb\cap \Dp$.
\begin{itemize}\item Using (\ref{MB}), we see that $g\in \M(\cb )$ and, hence,
$fg\in \cb$. \item Using \cite[Theorem $2.1$]{Vi} we deduce that
$g\in \M(\Dp )$ and, then it follows that $fg\in \Dp$.
\end{itemize} Since $\Dp\cap\cb \subset \Dq\cap \cb$, we have that
$fg\in\cb\cap \Dq $. Thus, we have proved that $g\in \M(\dpb,\dqb)$.
This finishes the proof of part (a).
\par Suppose now that $g$ is given by a power series with Hadamard
gaps and $g\in \M(\dpb,\dqb)$. Then $g\in \Dq$. Now, using
Theorem\,\@3.\,\@2 of \cite{GP:IE06}, we see that this implies that
$\mu_{g,q}$ is a Carleson measure. $\qed$
\end{proof}

\section{Multipliers on $\dph$}\label{mult-dph}
 \begin{Pf}{\em{Theorem \ref{Hinfty:pleq1}. }} Suppose that $0<p\le
 q<\infty $.
\par
If $g \in \M (\dph,\dqh)$ then, since  $\dph$ contains the constant
functions, it follows trivially that  $g\in\dqh$. \par On the other
hand, if $g \in \dqh$ and $f \in \dph$, it is clear that $gf\in
H^\infty $.  We also have
\begin{equation*}
\begin{aligned}
&\int _ \D \bigl|(g'f+gf')(z)\bigr|^q(1-|z|^2)^{q-1} d\,A(z)\\ &
\lesssim \int _ \D |(g'f)(z)|^q(1-|z|^2)^{q-1} \,dA(z)\,+ \int _ \D
|g(z)f'(z)|^q(1-|z|^2)^{q-1} d\,A(z)
\\ & \lesssim \|f\|^q_{H^\infty} \|g\|^q_{\Dq}  +  \|g\|^q_{H^\infty}
\|f\|^{q-p}_\cb\|f\|^{q-p}_{\Dp}\,<\infty .
\end{aligned}
\end{equation*}
Thus, $gf\in \Dq$ and, hence, $gf\in \dqh$. Consequently, we have
proved that  $g \in \M (\dph,\dqh)$.\,\,
 \end{Pf}
 \par\medskip
 \begin{Pf}{\em{Theorem \ref{Hinfty:q<p1}. }}
\par Let $\tilde{p}=\min\{p,1\}$ and $p^\star =\min\{p,2\}$. We shall split the proof in two cases.
\\ {\bf{Case} $\mathbf{1}$: $\mathbf{0<q<1}$. } Take a sequence $\{u_k\}_{k=1}^\infty \in
\ell^{\tilde{p}}\setminus\ell^q$ and let $f$ be defined by
$$f(z)=\sum_{k=1}^\infty u_kz^{2^k},\quad z\in \D.$$ Then, using Proposition \ref{Had} and the fact that $\tilde{p}\le 1$,
we see that $f\in \left (\Dp \cap H^\infty \right )\setminus \Dq$.
\\ {\bf{Case} $\mathbf{2}$: $\mathbf{1\le q<2}$.}
Let us consider a sequence $\{u_k\}$ such that $\{u_k\}_{k=1}^\infty
\in \ell^{p^\star}\setminus\ell^q$ and let choose $n_k=4^k$. We
claim that the function $\Phi\in H^\infty$ associated to $\{u_k\}$
and $\{n_k\}$ via Lemma \ref{le:fournier} satisfies that $\Phi\in
\dph\setminus\dqh.$
\par Arguing as in the proof of \cite[Lemma $1.6$ (i) ]{Vi} and bearing in mind Lemma \ref{le:fournier} (ii),  we deduce
$$||\Phi||^q_{\Dq}\gtrsim \sum_{k=0}^\infty \left|\widehat{\Phi(n_k)}\right|^q=||\{u_k\}||^q_{\ell^q}=\infty,$$
that is,  $\Phi\notin\Dq$.
\par By (\ref{1}), if $p\ge 2$ we are done. On the other hand, if $0<p<2$
by \cite[Theorem $1.1$ (ii)]{GP2}, M. Riesz theorem and
Lemma~\ref{le:fournier}~(iv),
\begin{equation*}\begin{split}
||\Phi||^q_{\Dp} & \le \int_0^1(1-r)^{p-1}M_2^p(r,\Phi')\,dr
\\ & \lesssim \sum_{k=0}^\infty \left(\left\Vert S_{2^k,2^{k+1}}\Phi\right\Vert_{H^2}\right)^{p/2}
\\ & \lesssim \sum_{k=0}^\infty \left(\left\Vert S_{4^k+1,4^{k+1}+1}\Phi\right\Vert_{H^2}\right)^{p/2}\lesssim ||\{u_k\}||^p_{\ell^p}<\infty,
\end{split}\end{equation*}
which finishes the proof.
 \end{Pf}
\par\medskip
\begin{Pf}{\em{Theorem \ref{Hinfty:q<p2}\,\@(a). }}
Assume that $0<q<1$, $0<q<p$ and that
 $g\in \M (\dph,\dqh)$ and $g\not\equiv 0$. Take $$f(z)=\sum_{k=1}^\infty \frac{z^{2^k}}{k^{\frac1{q}}}.$$
Then we use the Rademacher
functions as in the proof of Case\,\@3 of
Theorem\,\@\ref{th:blochdp}\,\@(ii) to get a contradiction. Hence,
$g\equiv 0$.

 \end{Pf}\par\medskip
 \begin{Pf}{\em{Theorem \ref{Hinfty:q<p2}\,\@(b). }}
Assume that $1\le q<2\le p$. By \cite[Theorem $1$]{Gi92} there is a function $f\in H^\infty $ such
that
$$\int_{0}^1(1-r^2)^{q-1} |f'(re^{i\theta})|^q\,dr = \infty \quad \text{for every $\theta\in B$}, $$
where $B$ is a subset of $[0, 2\pi ]$ whose Lebesgue measure $\vert
B\vert $ is $2\pi $. \par Suppose that $g \in \M (\dph,\dqh)$ and
$g\not\equiv 0$. Notice that $g\in \dqh$. Since
\begin{equation}
\begin{split}
\int_\D(1-|z|^2)^{q-1}|g'(z)f(z)|^q\,dA(z)\leq
\|f\|_{H^\infty}^q\|g\|_{\dqh }^q <\infty,
\end{split}
\end{equation}
it follows that
\begin{equation}\label{con:1}
\begin{split}
\int_\D(1-|z|^2)^{q-1}|g(z)f'(z)|^q\,dA(z)<\infty
\end{split}
\end{equation}
\par
 Since $g\in H^\infty$ and $g\not\equiv 0$,  there is a set $A=A(g) \subset [0, 2\pi ]$
 with $|A|>0$ and such that $\lim_{r\rightarrow 1^-}g(re^{i\theta})\not = 0$ if $\theta\in A$.
 Then, for every $\theta \in A\cap B$ there is $r_0(\theta) \in (0,1)$ such that $K=\inf \limits _ {r_0<r<1}|g(re^{i\theta})|
 >0$. Then
 $$\int_{0}^1(1-r^2)^{q-1}|g(re^{i\theta})|^q|f'(re^{i\theta})|^q\,dr\ge K^q \int_{r_0}^1(1-r^2)^{q-1}|f'(re^{i\theta})|^q\,dr = \infty,$$
 since $| A\cap B|>0$, this is in contradiction with (\ref{con:1}). Thus $g$ must be identically $0$. This finishes the proof.
\end{Pf}
\par\bigskip
\section{Some basic results on the space $BMOA_{\log}$}\label{Facts-BMOAlog}
\par We shall start this section by proving some embedding relations between $BMOA_{\log}$, $\cb_{\log}$ and $BMOA$.
With this aim, we recall that $g\in BMOA_{\log}$ if and only if
$$\sup_{a\in\D}\frac{\log^2\frac{2}{1-|a|}}{1-|a|}\int_{S(a)}|g'(z)|^2(1-|z|^2)\,dA(z)<\infty,$$
where $S(a)$ is the Carleson box associated to the interval
$$I_a=\left\{e^{it}\in\T:\, \left|\arg(a {e^{-it}}) \right|<\frac{1-|a|}{2}\right\},a\in\D\setminus\{0\},\quad I_0=\T.$$
\begin{proposition}\label{BMOAlogsubssetBlog}
If\, $1>\beta>\frac{1}{2}$, then $BMOA_{\log}\subsetneq
\cb_{\log}\subsetneq \cb_{log,\beta} \subsetneq BMOA$.
\end{proposition}
\begin{pf}
First, we  prove that $BMOA_{\log}\subset \cb_{\log}$. Take $f\in
BMOA_{\log}$. Let $a\in\D$ and assume without loss of generality
that $|a|>\frac12$. Set $a^\star=\frac{3|a|-1}{2}e^{i\arg a}$ so
that the disc $D\left(a,\frac{1-|a|}{2}\right)$ of center\, $a$\,
and radius $\frac{1-|a|}{2}$ is contained in the Carleson box
$S(a^\star)$. This inclusion together with the subharmonicity of
$|f'|^2$ and the fact that
 $(1-\vert z\vert )\asymp (1-\vert a\vert )$ for
$z\in D\left(a,\frac{1-|a|}{2}\right)$ gives
\begin{eqnarray*}&
\left (\log\frac{2}{1-\vert a\vert }\right )^2(1-\vert a\vert
)^2\vert f^\prime (a)\vert ^2 \lesssim \left (\log\frac{2}{1-\vert
a\vert }\right )^2\int_{D\left(a,\frac{1-|a|}{2}\right)}\vert f^\prime (z)\vert^2\,dA(z)\\
&\asymp \frac{\left (\log\frac{2}{1-\vert a\vert }\right )^2 }{1-\vert a\vert }
\int_{D\left(a,\frac{1-|a|}{2}\right)}(1-\vert z\vert ^2)\vert f^\prime (z)\vert
^2\,dA(z)\\
&\asymp \frac{\left (\log\frac{2}{1-\vert a^\star\vert }\right )^2 }{1-\vert a^\star\vert }
\int_{D\left(a,\frac{1-|a|}{2}\right)}(1-\vert z\vert ^2)\vert f^\prime (z)\vert
^2\,dA(z)\\
&\lesssim \frac{\left (\log\frac{2}{1-\vert a^\star\vert }\right )^2}{1-\vert a^\star\vert }
\int_{S(a^\star) }(1-\vert z\vert ^2)\vert f^\prime (z)\vert
^2\,dA(z),
\end{eqnarray*}
 so $f\in \cb_{\log }$.
\par Now, let us see that the inclusion is strict. We borrow ideas from \cite[Proposition $5.1$ (D)]{PelRat}.
Assume on the contrary to the assertion that $BMOA_{\log}=\cb_{\log}$.
 By  \cite[Theorem $1$]{GroPelRat}
(see also \cite{AD2012}) there are $g_1, g_2\in \cb_{\log}$ such that
$$|g_1'(z)|+|g'_2(z)|\gtrsim \frac{1}{(1-|z|)\log\frac{2}{1-\vert z\vert}},\quad z\in \D.$$
Then, for any $a\in\D$
\begin{equation*}\begin{split}
&\int_{S(a) }\frac{1}{(1-|z|)\log^2\frac{2}{1-\vert z\vert}}\,dA(z)
\lesssim \int_{S(a) }\left(|g_1'(z)|+|g'_2(z)|\right)^2(1-\vert z\vert ^2)\,dA(z)
\\ &\lesssim \int_{S(a)}|g_1'(z)|^2(1-\vert z\vert ^2)\,dA(z)
+\int_{S(a)}|g_2'(z)|^2(1-\vert z\vert ^2)\,dA(z)
\\ & \lesssim \frac{(1-|a|)}{\log^2\frac{2}{1-\vert a\vert}},
\end{split}\end{equation*}
so bearing in mind that
$$\int_{S(a) }\frac{1}{(1-|z|)\log^2\frac{2}{1-\vert z\vert}}\,dA(z)\asymp \frac{(1-|a|)}{\log\frac{2}{1-\vert a\vert}},$$
and letting $|a|\to 1^-$, we obtain a contradiction.
\par Assume now that $\beta\in \left(\frac{1}{2},1\right)$. Then it is clear that $\cb_{\log}\subsetneq \cb_{log,\beta}$. Furthermore,
$f(z)=\sum_{k=1}^\infty \frac{z^{2^k}}{k^\beta}\in \cb_{log,\beta}\setminus \cb_{log}$ (see \cite[p. $20$]{PelRathg})
\par
The inclusion $\cb_{\log,\beta}\subsetneq BMOA$, for
$\beta>\frac{1}{2}$, follows easily using the characterization of
$BMOA$ in terms of Carleson measures (see \cite[p.\,\@669]{GaGiHe}).
Finally, we observe that $f(z)=\log\frac{1}{1-z}\in BMOA\setminus
\cb_{\log,\beta}$ for any $\beta>0$. This concludes the proof.
\end{pf}
\par\medskip
Next we find a simple sufficient condition for the membership a a
function $f\in\hol (\D )$ in the space $BMOA_{\log }$.
\begin{proposition}\label{suf-con-BMOAlog}
Let $f$ be an analytic function in $\D $. If
\begin{equation}\label{suf-cond}\int_0^1(1-r)\left (\log\frac{1}{1-r}\right )^2\left[M_\infty
(r,f^\prime )\right]^2\,dr<\infty \end{equation} then $f\in
BMOA_{\log }$.
\end{proposition}
\begin{pf}
Suppose that $f$ satisfies (\ref{suf-cond}) Let $I$ be an interval
in $\T$ of length $h$, say $I=\{ e^{it} : \theta_0<t<\theta_0+h\} $.
Then
\begin{eqnarray*}&\frac{\left (\log\frac{2}{\vert I\vert }\right)^2}{\vert I\vert
} \int _{S(I)}(1-\vert z\vert ^2)\vert f^\prime (z)\vert ^2 \,dA(z)
\asymp \frac{\left
(\log\frac{2}{h}\right)^2}{h}\int_{1-h}^1\int_{\theta_0}^{\theta_0+h}(1-r)\vert
f^\prime (re^{it})\vert ^2\,dr \\& \lesssim
\left(\log\frac{2}{h}\right)^2 \int_{1-h}^1(1-r)\left[M_\infty
(r,f^\prime )\right]^2\,dr\le\int_{1-h}^1(1-r)\left[M_\infty
(r,f^\prime )\right]^2\left (\log\frac{2}{1-r}\right )^2\,dr\\& \le
\int_{0}^1(1-r)\left[M_\infty (r,f^\prime )\right]^2\left
(\log\frac{2}{1-r}\right )^2\,dr.
\end{eqnarray*}
\end{pf}
\par\medskip
Now we turn to the question of finding conditions on the Taylor
coefficients of a function $f\in \hol (\D )$ enough to assert that
$f\in BMOA_{\log }$. We shall need two lemmas. The first one
estimates an integral which may be viewed as a generalization of the
classical beta function (compare with Lemma\,\@2 of \cite{D:Random})
and we omit its proof.
\begin{lemma}\label{like-Beta} Whenever $m=1, 2, 3, \dots $ and $\alpha
>0$, we have
\begin{equation}\label{int-with-gamma=2} \int_0^1x^n(1-x)^m\left
(\log \frac{1}{1-x}\right )^\alpha \,dx\,\asymp\,\frac{(\log
n)^\alpha }{n^{m+1}},\quad\text{as $n\to\infty $}.
\end{equation}
\end{lemma}
\par\medskip
\begin{lemma}\label{lema-coef} Suppose that $\alpha >0$ and
let $g$ be an analytic function in $\D $, $g(z)=\sum_{n=0}^\infty
a_nz^n$ ($z\in \D$). The following two conditions are equivalent:
\par (i)\,\, $\int_{\D}(1-\vert z\vert ^2)\vert g^\prime (z)\vert
^2\left (\log\frac{2}{1-\vert z\vert }\right )^\alpha\,dA(z)<\infty
.$
\par (ii)\,\, $\sum_{n=1}^\infty \vert a_n\vert ^2[\log n]^\alpha <\infty
$.\end{lemma}
\begin{pf} We have
\begin{eqnarray*}&\int_{\D}(1-\vert z\vert ^2)\vert g^\prime (z)\vert
^2\left (\log\frac{2}{1-\vert z\vert }\right )^\alpha \,dA(z)\asymp
\int_0^1r(1-r)M_2(r,g^\prime )^2\left (\log\frac{2}{1-\vert z\vert
}\right )^ \alpha\,dr\\ & =\sum_{n=1}^\infty n^2\vert a_n\vert
^2\int _0^1(1-r)r^{2n-1}\left (\log\frac{2}{1-\vert z\vert }\right
)^\alpha\,dr
\end{eqnarray*}
Now, using Lemma\,\@\ref{like-Beta} with $m=1$ we see that $\int
_0^1(1-r)r^{2n-1}\left (\log\frac{2}{1-\vert z\vert }\right
)^\alpha\,dr\asymp \frac{[\log n]^\alpha}{n^2}$. Then it follows
that $\int_{\D}(1-\vert z\vert ^2)\vert g^\prime (z)\vert ^2\left
(\log\frac{2}{1-\vert z\vert }\right )^2\,dA(z)\asymp
\sum_{n=1}^\infty \vert a_n\vert ^2[\log (n+1)]^\alpha $. \end{pf}
\par\medskip
We close this section proving Proposition\,\@\ref{lac-bmoalog-log3}.
\par \begin{Pf}{\,\em{Proposition\,\@\ref{lac-bmoalog-log3}.}}\,
Suppose that $\sum_{k=0}^\infty \vert a_k\vert ^2(\log n_k)^3<\infty
$ and
$$\hbox{$f(z)=\sum_{k=0} \sp\infty a_kz\sp {n_k}$\, ($z\in \D $)\,
with $n_{k+1}\ge \lambda n_k$\, for all $k$, and $\lambda
>1$.}$$
Using the Cauchy–-Schwarz inequality and the fact that $\sum_{k=0}
\sp\infty r^{2n_k}\lesssim \log\frac{2}{1-r}$ (because the function
$h$ given by  $h(z)=\sum z^{2n_k}$ is a Bloch function), we see that
\begin{eqnarray*}&
[rM_\infty (r, f^\prime )]^2\le \left (\sum_{k=0} \sp\infty n_k\vert
a_k\vert r^{n_k}\right )^2\\&\le \left (\sum_{k=0} \sp\infty
n_k^2\vert a_k\vert ^2r^{2n_k}\right )\left (\sum_{k=0} \sp\infty
r^{2n_k}\right )\lesssim \left (\log\frac{2}{1-r}\right )\sum_{k=0}
\sp\infty n_k^2\vert a_k\vert ^2r^{2n_k}.\end{eqnarray*} Then, using
Lemma\,\@\ref{like-Beta} with $m=1$ and $\alpha =3$, we obtain
\begin{eqnarray*}&
\int_0^1(1-r)\left (\log\frac{1}{1-r}\right )^2\left [M_\infty
(r,f^\prime )\right ]^2\,dr\lesssim \int_0^1(1-r)\left
(\log\frac{1}{1-r}\right )^3\left (\sum_{k=0} \sp\infty n_k^2\vert
a_k\vert ^2r^{2n_k}\right )\,dr\\& =\sum_{k=0} \sp\infty n_k^2\vert
a_k\vert ^2\int_0^1r^{2n_k}(1-r)\left (\log\frac{1}{1-r}\right
)^3\,dr \lesssim \sum_{k=0} \sp\infty \vert a_k\vert ^2\left (\log
n_k\right )^3<\infty .\end{eqnarray*} Then
Proposition\,\@\ref{suf-con-BMOAlog} implies that $f\in BMOA_{\log
}$.
\par To see that $f\in H^\infty $ observe that $\lambda^k\lesssim
n_k$ and $\vert a_k\vert ^2\lesssim (\log n_k)^{-3}$. Then it
follows that $\vert a_k\vert =\og \left (k^{-3/2}\right )$, as
$k\to\infty $ and the result follows.
\end{Pf}
\par\bigskip
\section{Random power series}\label{Random-ps}\par\medskip
In this section we shall consider random power series analytic in
$\D $ of the form $$\sum_{n=0}^\infty \epsilon_na_nz^n$$ where the
$\epsilon_n$'s are random signs. More precisely, if $f\in\hol(\D )$,
$f(z)=\sum_{n=0}^\infty a_nz^n$ ($z\in \D $), we set
$$f_t(z)=\sum_{n=0}^\infty r_n(t)a_nz^n,\quad 0\le t<1,\quad z\in
\D ,$$ where the ${r_n}$'s are the Rademacher functions. Each
function $f_t$ is analytic in $\D $. Littlewood \cite{Li} (see also
\cite[Appendix\,\@A]{D}) proved that if $\sum_{n=0}^\infty \vert
a_n\vert^2<\infty $ then $f_t\in \cap_{0<p<\infty }H^p$ almost
surely (a.\,\@s.), that is, for almost every $t$. On the other hand,
the condition $\sum_{n=0}^\infty \vert a_n\vert^2=\infty $ implies
that for almost every $t$, $f_t$ has a radial limit almost nowhere.
\par Paley and Zygmund \cite{PaZy1930} gave an example
of an $f$ with
\begin{equation}\label{an2logn}\sum_{n=1}^\infty\vert a_n\vert^2\log n<\infty
\end{equation}
such that $f_t\notin H^\infty $ for every $t$.
\par Anderson, Clunie and Pommerenke \cite{ACP} used a result of Salem and Zygmund \cite{SaZy}
on the behaviour of the maxima of the partial sums of random
trigonometric series to prove that (\ref{an2logn}) implies that
$f_t\in \cb$ a.\,\@s. and that this condition is best possible.
Later on, Sledd \cite{Sledd} used also the Salem and Zygmund theorem
to show that (\ref{an2logn}) actually implies that $f_t\in BMOA$
a.\,\@s.
\par
Duren proved in \cite{D:Random} the following result.
\begin{other}\label{duren-ran} If\, $0\le \beta \le 1$\, and\, $\sum_{n=1}^\infty \vert a_n\vert ^2(\log n)^\beta <\infty $, then
for almost every\, $t\in [0, 1]$, the function
$$f_t(z)=\sum_{n=1}^\infty r_n(t)a_nz^n,\quad z\in \D,$$ satisfies
\begin{equation}\label{Du-ft-int-m-infty}
\int_0^1(1-r)\left (\log \frac{1}{1-r}\right )^{\beta -1}\left
[M_\infty (r, f_t^\prime )\right ]^2\,dr <\infty .\end{equation}
\end{other}
\par\medskip Using this, Duren gave in \cite{D:Random} a new proof
of Sledd's theorem. Next we prove an analogue of Duren's theorem for
$\beta =3$. This will allow us to obtain the analogue of Sledd's
theorem for $BMOA_{\log }$.
\begin{theorem}\label{random.likeduren-bmoalog}
If\, $\sum_{n=1}^\infty \vert a_n\vert ^2(\log n)^3<\infty $ then
for almost every $t\in [0, 1]$, the function
$$f_t(z)=\sum_{n=1}^\infty r_n(t)a_nz^n,\quad z\in \D,$$ satisfies
\begin{equation}\label{ft-int-m-infty}
\int_0^1(1-r)\left (\log \frac{1}{1-r}\right )^2\left [M_\infty (r,
f_t^\prime )\right ]^2\,dr <\infty .\end{equation}
\end{theorem}
\par\medskip
Another result of \cite{SaZy} implies that the condition
$\sum_{n=1}^\infty \vert a_n\vert^2[\log n]^\beta <\infty $ for some
$\beta >1$, implies for almost every $t$, $f_t$ has a continuous
extension to the closed unit disc. Using this,
Proposition\,\@\ref{suf-con-BMOAlog}, and
Theorem\,\@\ref{random.likeduren-bmoalog} we obtain the first part
of Theorem~\ref{random.likesledd-bmoalog}. Part~(ii) of this theorem
can be proved arguing as in section\,\@3.\,\@4 of \cite{ACP}, and we
omit the proof.
\par\medskip
The proof of Theorem\,\@\ref{random.likeduren-bmoalog} follows the
lines of that of Theorem\,\@\ref{duren-ran} in \cite{D:Random}. We
shall use the result of Salem and Zygmund already mentioned
(Lemma\,\@1 of \cite{D:Random}), Hilbert's inequality (Lemma\,\@2 of
\cite{D:Random}) and Lemma\,\@\ref{like-Beta} with $m=3$ and $\alpha
=2$.
\par\medskip
\begin{Pf}{\,\em{Theorem \ref{random.likeduren-bmoalog}.}}\,
Set $$B_n^2=\sum_{k=1}^\infty k^2\vert a_k\vert ^2,\quad  n=1, 2,
\dots ,$$ and $\psi (r)=(1-r)\sum_{n=1}^\infty B_n\sqrt {\log n}r^n$
($0<r<1$). Just as in p.\,\@84 of \cite{D:Random}, we have
\begin{equation}\label{ineq-f-prime-duren}
\vert {f_t}^\prime (z)\vert \le C\psi (r), \quad \vert z\vert
=r,\quad 0<r<1,\quad\text{almost surely}.\end{equation} Using
Lemma\,\@\ref{like-Beta}, the simple fact that $\frac{\log
x}{x^{3/2}}$ decreases as $x$ increases in $[e^{2/3},\infty)$, and
Hilbert's inequality, we deduce
\begin{align*}&
\int_0^1(1-r)\left (\log \frac{1}{1-r}\right )^2[\psi
(r)]^2\,dr\notag
\\& \asymp \int_0^1(1-r)^3\left (\log\frac {1}{1-r}\right )^2\left
[\sum_{n=1}^\infty B_n\sqrt {\log n}r^n\right ]^2\,dr\notag\\&
=\sum_{n=1}^\infty \sum_{j=1}^\infty B_n\sqrt {\log n}B_j\sqrt {\log
j}\int_0^1r^{n+j}(1-r)^3\left (\log \frac{1}{1-r}\right
)^2\,dr\notag\\& \lesssim \sum_{n=1}^\infty \sum_{j=1}^\infty
\frac{B_n\sqrt {\log n}B_j\sqrt {\log j}}{(n+j)^4}[\log
(n+j)]^2\tag{*}
%\\& \lesssim \sum_{n=1}^\infty \sum_{j=1}^\infty \frac{B_n[\log
%(n+10)]^{5/2}B_j[\log (j+10)]^{5/2}}{(n+j)^4}\label{a}\notag
\\& \le \sum_{n=1}^\infty \sum_{j=1}^\infty \frac{1}{n+j}\frac{B_n[\log
n]^{3/2}}{n^{3/2}}\frac{B_j[\log j]^{3/2}}{j^{3/2}}\notag
\\& \lesssim
\sum_{n=1}^\infty \vert B_n\vert ^2\frac{[\log
n]^{3}}{n^3}.\notag
\end{align*}
Now,
\begin{align*}&\sum_{n=1}^\infty \vert B_n\vert ^2\frac{[\log
n]^{3}}{n^3}=\sum_{n=1}^\infty \sum_{k=1}^nk^2\vert a_k\vert
^2\frac{[\log n]^{3}}{n^3}\tag{**}\\ &=\sum_{k=1}^\infty
k^2\vert a_k\vert ^2\sum_{n=k}^\infty \frac{[\log
n]^{3}}{n^3}\asymp\sum_{k=1}^\infty \vert a_k\vert^2[\log
k]^{3}<\infty.\notag
\end{align*}
Then (\ref{ineq-f-prime-duren}), (*) and (**) imply that
(\ref{ft-int-m-infty}) holds for almost every $t$, finishing the
proof.
\end{Pf}

\par\bigskip
\section{Multipliers on $\Dpbmoa$}\label{multbmoa}
In this section we shall prove our results concerning multipliers
from $\Dpbmoa $ to $\Dqbmoa $. Let start with the following result.
\begin{theorem}\label{inclusion-always}
For any $p, q$ with $0<p, q<\infty $ we have
$$\M(\Dpbmoa ,\Dqbmoa )\subset BMOA_{\log}\cap H^\infty =\M(BMOA).$$
\end{theorem}
\begin{pf}
The proof uses arguments similar to those in that of
Theorem\,\@\ref{th:blochdp}\,\@(i) and, hence, we shall omit some
details.
\par Using that the family $\{ \varphi _a : a\in \D \} $ is bounded
in $\mathcal D^\lambda _{\lambda -1}\cap BMOA$ for all $\lambda >0$,
we deduce that
$$\M(\Dpbmoa ,\Dqbmoa )\subset H^\infty ,\quad 0<p, q<\infty .$$
\par Suppose now that $0<p, q<\infty $ and $g\in \M(\Dpbmoa ,\Dqbmoa
)$. Let us use the test functions $f_a$ ($a\in \D $) defined by
$$f_a(z)=\log \frac{1}{1-\bar az},\quad z\in \D .$$
It is easy to see that the family $\{ f_a : a\in \D \} $ is also
bounded $\mathcal D^\lambda _{\lambda -1}\cap BMOA$ for all $\lambda
>0$.  On the other hand, there exists an
absolute constant $C>0$ such that for any arc $I\subset \partial \D$
$$\frac{1}{C}\log\frac{2}{\vert I\vert }\le \vert f_a(z)\vert \le C\log\frac{2}{\vert I\vert
},\quad z\in S(I),$$ where  $a=(1-\frac{\vert I\vert } {2\pi })\xi $
with $\xi $ the center of $I$. \par
 Then we have
\begin{eqnarray*}&
\frac{\log^2\frac{2}{\vert I\vert }}{\vert I\vert
}\int_{S(I)}(1-\vert z\vert ^2 )\vert g^\prime (z)\vert
^2\,dA(z)\,\le \frac{C^2}{\vert I\vert }\int_{S(I)}(1-\vert z\vert
^2 )\vert f_a(z)\vert ^2\vert g^\prime (z)\vert ^2\,dA(z)\\ \lesssim
& \frac{C^2}{\vert I\vert }\int_{S(I)}(1-\vert z\vert ^2 )\vert
(f_ag)^\prime (z)\vert ^2\,dA(z)\,+\, \frac{C^2}{\vert I\vert
}\int_{S(I)}(1-\vert z\vert ^2 )\vert f_a^\prime (z)\vert ^2\vert
g(z)\vert ^2\,dA(z).
\end{eqnarray*}
Since $g\in \M(\Dpbmoa ,\Dqbmoa )$, the family $\{ f_ag : a\in \D\}
$ is bounded in $BMOA$ and hence \,$\sup_I\frac{C^2}{\vert I\vert
}\int_{S(I)}(1-\vert z\vert ^2 )\vert (f_ag)^\prime (z)\vert
^2\,dA(z)<\infty $.\, Also, using that $g\in H^\infty $ and that the
family $\{ f_a :a\in \D\} $ is bounded in $BMOA$, we deduce
that\newline $\sup_I\frac{C^2}{\vert I\vert }\int_{S(I)}(1-\vert
z\vert ^2 )\vert f_a^\prime (z)\vert ^2\vert g(z)\vert
^2\,dA(z)<\infty $.\, Consequently, we have that
$$\sup_I\frac{\log^2\frac{2}{\vert I\vert }}{\vert I\vert
}\int_{S(I)}(1-\vert z\vert ^2 )\vert g^\prime (z)\vert
^2\,dA(z)<\infty , $$ that is, $g\in BMOA_{\log}$.
\end{pf}

\par\medskip
\begin{Pf}{\,\em{Theorem \ref{BMOA1<qpleq}.}}\, Suppose that $1<q<\infty $ and $0<p\le q<\infty $.
In view of Theorem\,\@\ref{inclusion-always}, we only have to prove
that $\M(BMOA)\subset \M(\Dpbmoa ,\Dqbmoa )$. \par Take $g\in
\M(BMOA)$ and $f\in BMOA \cap \Dp$. Then, clearly, $fg\in BMOA$.
Using Lemma~\ref{inclusionDpB} and the closed graph theorem, we
obtain
\begin{equation}
\begin{aligned}\label{eq:dppart}
&\int _\D |(fg)'(z)|^{q}(1-|z|^2)^{q-1} dA(z)
\\ & \lesssim \int _\D |f'(z)g(z)|^{q}(1-|z|^2)^{q-1}dA(z)
+ \int _\D |g'(z)f(z)|^{q}(1-|z|^2)^{q-1}dA(z)
\\ & \lesssim \|g\|_{H^\infty}^q\|f\|_{\Dpbmoa}^q
+ \int _\D |f(z)g'(z)|^{q}(1-|z|^2)^{q-1}dA(z).
\end{aligned}
\end{equation}
Now, Proposition\,\@\ref{BMOAlogsubssetBlog} implies that $g\in
\mathcal B_{\log }$. Also, since $BMOA\subset H^q$, we have that
$f\in H^q$. Then we see that the last integral in (\ref{eq:dppart})
is finite as in the proof of (\ref{eq:dppart2hinfty}). Thus, we have
proved that $g\in \M(\Dpbmoa ,\Dqbmoa )$ finishing the proof.
\end{Pf}
\par\medskip
\begin{Pf}{\,\em{Theorem \ref{BMOAq<p}.}}\,
Suppose that $0<q<p<\infty $, $q<2$ and $g\in \M(\Dpbmoa ,\Dqbmoa
)$ with $g\not\equiv 0$. Take $a_n=\frac{1}{n^\lambda }$ ($n=1, 2, \dots
$) with
$$\max \left (\frac{1}{2}, \frac{1}{p}\right )<\lambda \le
\frac{1}{q}$$ and set $f(z)=\sum_{n=1}^\infty a_nz^{2^n}$ ($z\in \D$). We
have that $f\in \Dpbmoa \setminus \Dq$. Then we use the Rademacher
functions as in the proof of Case\,\@3 of
Theorem\,\@\ref{th:blochdp}\,\@(ii) to get a contradiction. Hence,
$g\equiv 0$.
\end{Pf}
\par\medskip
Now we turn to prove Theorem~\ref{randon-bmoalogMdp}. Let us notice
that (i) follows from Theorem\,\@\ref{random.likeduren-bmoalog} and
(ii) from Theorem\,\@\ref{random.likesledd-bmoalog}\,\@(i). To prove
(iii) we shall use the following lemma.
\begin{lemma}\label{hinfty+bmoalog+som}
Suppose that $0<q<2$ and $\alpha >0$. Let $f$ be an analytic
function in $\D $ of the form $f(z)=\sum_{n=0}^\infty a_nz^n$ ($z\in
\D$), with $\sum_{n=0}^\infty \vert a_n\vert ^2[\log n]^\alpha
<\infty $. If $f\in BMOA_{\log }\cap H^\infty $ then
$$f\in \M(\Dpbmoa ,\Dqbmoa ),\quad\text{whenever\, $0<p\le q$\,
and\, $\frac{q\alpha }{2-q}>1$.}$$
\end{lemma}
\par\medskip For $0<q\le 1$, (iii) of
Theorem~\ref{randon-bmoalogMdp} follows using (ii) and the lemma
with $\alpha =3$, while, for $1<q<\infty $, it follows from Theorem
\ref{BMOA1<qpleq}.
\par\medskip \begin{Pf}{\,\em{Lemma\,\@\ref{hinfty+bmoalog+som}.}}\,
Suppose that $f$ is in the conditions of the lemma and that $0<p\le
q$\, and\, $\frac{q\alpha }{2-q}>1$.
\par Take $h\in \Dpbmoa $.
 Since $BMOA_{\log }\cap H^\infty =\M (BMOA)$, it follows that
$fh\in BMOA$. \par We have also
\begin{eqnarray*}
&\int_{\D}(1-\vert z\vert )^{q-1}\vert (fh)^\prime (z)\vert ^q\,dA(z)\\&
\lesssim \int_{\D}(1-\vert z\vert )^{q-1}\vert f(z)\vert ^q\vert
h^\prime (z)\vert ^q\,dA(z) + \int_{\D}(1-\vert z\vert )^{q-1}\vert
f^\prime(z)\vert ^q\vert h(z)\vert ^q\,dA(z)\\&=I_1+I_2.
\end{eqnarray*}
The first summand $I_1$ is finite because $f\in H^\infty $ and $h\in
\Dpbmoa\subset \Dqbmoa $. Let us estimate the second one $I_2.$
Using H\"{o}lder's inequality with the exponents $\frac{2}{q}$ and
$\frac{2}{2-q}$, we obtain
\begin{eqnarray*}
&I_2=\int_{\D}(1-\vert z\vert )^{q-1}\vert f^\prime(z)\vert ^q\vert
h(z)\vert ^q\,dA(z)\\&= \int_{\D}\vert f^\prime(z)\vert ^q(1-\vert
z\vert )^{q/2}\left (\log \frac{e}{1-\vert z\vert }\right )^{\frac{q\alpha}{2}
}\left (\log \frac{e}{1-\vert z\vert }\right )^{-\frac{q\alpha}{2}}\vert
h(z)\vert ^q1-\vert z\vert )^{\frac{q}{2}-1}\,dA(z)\\& \le \left
[\int_{\D}\vert f^\prime(z)\vert ^2(1-\vert z\vert )\left (\log
\frac{e}{1-\vert z\vert }\right )^{\alpha }\,dA(z) \right ]^{q/2}
\left [\int_{\D}\frac{\vert h(z)\vert ^{\frac{2q}{2-q}}}{\left (\log
\frac{e}{1-\vert z\vert }\right )^{\frac{q\alpha }{2-q}}(1-\vert
z\vert )}\,dA(z) \right ]^{(2-q)/q}.
\end{eqnarray*}
Using Lemma\,\@\ref{lema-coef}, it follows that the first integral
in the last product is finite. Now, notice that $f\in H^\lambda $
for all $\lambda <\infty $ to deduce
\begin{eqnarray*}
\int_{\D}\frac{\vert h(z)\vert ^{\frac{2q}{2-q}}}{\left (\log
\frac{e}{1-\vert z\vert }\right )^{\frac{q\alpha }{2-q}}(1-\vert
z\vert )}\,dA(z)\le \Vert f\Vert _{H^\frac{2q}{2-q}}^{\frac{2q}{2-q}}
\int_0^1\frac{1}{\left (\log \frac{e}{1-r}\right )^{\frac{q\alpha
}{2-q}}(1-r)}\,dr
\end{eqnarray*}
and this integral is finite because $\frac{q\alpha }{2-q}>1$. Thus
$I_2<\infty $. Then we have that $fh\in \Dq $ and, hence, $fh\in
\Dqbmoa $. Consequently, we have proved that $f\in \M(\Dpbmoa
,\Dqbmoa )$.
\end{Pf}
\par\medskip To finish the proof of Theorem~\ref{randon-bmoalogMdp}
take $q\in (0, 1/2)$ and let $\{ a_n\} $ be defined as follows:
$$a_{2^k}=(k+1)^{-1/q},\quad k=0, 1,\dots $$
and $a_n=0$, if $n$ is not a power of $2$. Set
$$f(z)=\sum_{n=0}^\infty a_nz^n=\sum_{k=0}^\infty
(k+1)^{-1/q}z^{2^k},\quad z\in \D .$$ It is clear that $\{ a_n\} $
satisfies (\ref{rand-log3}). Furthermore, for almost every $t$,
$f_t$ is given by a lacunary power series, $f_t(z)=\sum_{k=0}^\infty
r_{2^k}(t)a_{2^k}z^{2^k}$, which does not belong to $\Dq$ because
$\sum_{k=0}^\infty \vert a_{2^k}\vert^q=\infty $.

\par\bigskip
Now  turn to consider multipliers in $\M (\Dpbmoa ,\Dqbmoa )$ given
by power series with Hadamard gaps.  First we show that whenever
$0<p\le q\le 1$ the power series with Hadamard gaps in $\M (\Dpbmoa
,\Dqbmoa )$ coincide with those in $\Dq \cap BMOA_{\log }$.
\begin{theorem}\label{lac-qle10<pleq}
Suppose that $0<p\le q\le 1$ and let $g$ be an analytic function in
$\D $ given by a power series with Hadamard gaps. Then the following
conditions are equivalent:
\begin{itemize}\item[(a)] $g\in \M (\Dpbmoa ,\Dqbmoa )$.
\item[(b)] $g\in \Dq \cap BMOA_{\log }$.
\end{itemize}
\end{theorem}
\begin{pf}\,
Since $\Dpbmoa $ contains the constants functions, it is clear that
$\M(\Dpbmoa ,\Dqbmoa )\subset \Dq$ and the inclusion
\newline $\M(\Dpbmoa ,\Dqbmoa )\subset BMOA_{\log }$ follows from
Theorem\,\@\ref{inclusion-always}. Hence, the implication\, (a)
$\Rightarrow $ (b)\, holds.
\par Let us prove next the other implication. So take $g\in
\Dq\cap BMOA_{\log}\cap \mathcal L$,
$$\hbox{$g(z)=\sum_{k=0} \sp\infty a_kz\sp {n_k}$\, ($z\in \D $)\,
with $n_{k+1}\ge \lambda n_k$\, for all $k$, for a certain $\lambda
>1$.}$$
We have $\sum_{k=0}^\infty \vert a_k\vert ^q<\infty $ which, since
$q\le 1$, implies that $\sum_{k=0}^\infty \vert a_k\vert <\infty $.
Thus $g\in H^\infty $. Then $g\in BMOA_{\log }\cap H^\infty
=\M(BMOA)$. \par Take $f\in \Dpbmoa $.
%Notice than then $g\in \Dq$.
Since $g\in \M(BMOA)$, we have that $gf\in BMOA$. Now,
\begin{eqnarray*}
&\int_{\D }\vert (gf)^\prime (z)\vert ^q(1-\vert z\vert
^2)^{q-1}\,dA(z)\\ &\lesssim \int_{\D }\vert g(z)\vert ^q\vert
f^\prime (z)\vert ^q(1-\vert z\vert ^2)^{q-1}\,dA(z)\,+\,\int_{\D
}\vert f(z)\vert ^q\vert g^\prime (z)\vert ^q(1-\vert z\vert
^2)^{q-1}\,dA(z)
\\ &=I_1+I_2.
\end{eqnarray*} The first summand $I_1$ is finite because $g\in H^\infty $
and $f\in \Dq$.\par Let us estimate the second one. Using
Theorem\,\@3.\,\@2 of \cite{GP:IE06} we see that the measure $\mu_
{g,q}$ in $\D $ defined by $d\mu_ {g,q}(z)=(1-\vert z\vert
^2)^{q-1}\vert g^\prime (z)\vert ^q\,dA(z)$ is a Carleson measure
and (see, e.\,\@g., \cite[Theorem\,\@1]{Wu} or
\cite[Theorem~2.\,\@1]{Vi}) this implies that $\mu_ {g,q}$ is a
Carleson measure for $\Dq $, that is, $\Dq \subset L^q(d\mu_
{g,q})$. Hence $f\in L^q(d\mu_ {g,q})$ which is equivalent to saying
that $I_2<\infty $. Hence, $gf\in \Dq $.
\par So, we have proved that $gf\in \Dqbmoa $ for any $f\in \Dpbmoa
$, that is, $g\in \M(\Dpbmoa , \Dqbmoa )$.
\end{pf}
\par\bigskip

Finally, we obtain also the analogue of
Theorem\,\@\ref{randon-bmoalogMdp} for lacunary power series.
\begin{theorem}\label{lac-bmoalogMdp}
Let $f\in \hol (\D )$ be given by a lacunary power series, of the
form
$$\hbox{$f(z)=\sum_{k=0} \sp\infty a_kz\sp {n_k}$\, ($z\in \D $)\,
with $n_{k+1}\ge \lambda n_k$\, for all $k$, for a certain $\lambda
>1$},$$
and suppose that the sequence of coefficients $\{ a_k\}
_{k=0}^\infty $ satisfies
\begin{equation}\label{lac-log3}\sum_{k=1}^\infty \vert a_k\vert^2(\log
n_k)^3<\infty .\end{equation} Then the function $f$ satisfies the
following conditions:
\begin{itemize}
\item[(i)] $f\in \M(BMOA )$
%$f\in BMOA_{\log }\cap H^\infty $.
\item[(ii)] $f\in M(\Dpbmoa , \Dqbmoa )$ whenever $0<p\le q$ and
$q>\frac{1}{2}$.
\end{itemize}
\par Furthermore, if $0<q<\frac{1}{2}$ then there exists a
sequence $\{ a_k\} $ which satisfies (\ref{lac-log3}) and such that
$f\notin \Dq$. Thus for this sequence $\{ a_k\} $ the function $f$
satisfies:
\begin{itemize}
\item[(a)] $f\in \M(BMOA )$.
\item[(b)] If\, $0<p\le \lambda $ and $\lambda >1/2$ then $f\in M(\Dpbmoa , \mathcal D^\lambda_{\lambda -1} )$ whenever $0<p\le \lambda $.
\item[(c)]  $f\notin M(\Dpbmoa , \Dqbmoa )$ whenever $0<p\le q$.
\end{itemize}
\end{theorem}
\begin{pf}\, Part (i) follows from Proposition\,\@\ref{lac-bmoalog-log3}.
Part (ii) follows from Theorem \ref{BMOA1<qpleq} for $q>1$  and  from Lemma\,\@\ref{hinfty+bmoalog+som} (with
$\alpha =3$) whenever $0<q\le 1$.
\par\medskip Now, if $0<q<\frac{1}{2}$ take $$a_k=k^{-1/q},\quad
k=1, 2, \dots $$ and $$f(z)=\sum_{k=1}^\infty a_kz^{2^k},\quad z\in
\D .$$ Clearly,
$$
\sum_{k=1}^\infty \vert a_k\vert ^2k^3<\infty ,\quad\text{and}\quad
\sum_{k=1}^\infty \vert a_k\vert ^q=\infty .$$ Then $f$ satisfies
conditions (a), (b) and (c) of Theorem\,\@\ref{lac-bmoalogMdp}.
\end{pf}
\par\medskip As we mentioned in Section~\ref{intro}, Theorem\,\@\ref{lac-bmoalogMdp}
also shows that Theorem\,\@\ref{BMOA1<qpleq} does not remain true
for $q<1/2$.

\end{document}